\documentclass[a4paper,onecolumn,english]{article}

\usepackage[T1]{fontenc}
\usepackage[utf8]{inputenc}
\usepackage{babel}
\usepackage{pdflscape}
\usepackage{enumitem}
\usepackage{amsmath}
\usepackage{amsfonts}
\usepackage{amssymb}
\usepackage{amsthm}
\usepackage{thmtools}
\usepackage{dsfont}
\usepackage{bm}
\usepackage{relsize}
\usepackage{adjustbox}
\usepackage{graphicx}
\usepackage[title]{appendix}
\usepackage[left=2.5cm,right=2.5cm,top=2.5cm,bottom=2.5cm]{geometry}
\usepackage[colorlinks=true,
            linkcolor = blue,
            urlcolor  = blue,
            citecolor = red,
			breaklinks]{hyperref}
\allowdisplaybreaks

\numberwithin{equation}{section}

\makeatletter
	\let\save@mathaccent\mathaccent
	\newcommand*\if@single[3]{%
	  \setbox0\hbox{${\mathaccent"0362{#1}}^H$}%
	  \setbox2\hbox{${\mathaccent"0362{\kern0pt#1}}^H$}%
	  \ifdim\ht0=\ht2 #3\else #2\fi
	  }
	\newcommand*\rel@kern[1]{\kern#1\dimexpr\macc@kerna}
	\newcommand*\widebar[1]{\@ifnextchar^{{\wide@bar{#1}{0}}}{\wide@bar{#1}{1}}}
	\newcommand*\wide@bar[2]{\if@single{#1}{\wide@bar@{#1}{#2}{1}}{\wide@bar@{#1}{#2}{2}}}
	\newcommand*\wide@bar@[3]{%
	  \begingroup
	  \def\mathaccent##1##2{%
	    \let\mathaccent\save@mathaccent
	    \if#32 \let\macc@nucleus\first@char \fi
	    \setbox\z@\hbox{$\macc@style{\macc@nucleus}_{}$}%
	    \setbox\tw@\hbox{$\macc@style{\macc@nucleus}{}_{}$}%
	    \dimen@\wd\tw@
	    \advance\dimen@-\wd\z@
	    \divide\dimen@ 3
	    \@tempdima\wd\tw@
	    \advance\@tempdima-\scriptspace
	    \divide\@tempdima 10
	    \advance\dimen@-\@tempdima
	    \ifdim\dimen@>\z@ \dimen@0pt\fi
	    \rel@kern{0.6}\kern-\dimen@
	    \if#31
	      \overline{\rel@kern{-0.6}\kern\dimen@\macc@nucleus\rel@kern{0.4}\kern\dimen@}%
	      \advance\dimen@0.4\dimexpr\macc@kerna
	      \let\final@kern#2%
	      \ifdim\dimen@<\z@ \let\final@kern1\fi
	      \if\final@kern1 \kern-\dimen@\fi
	    \else
	      \overline{\rel@kern{-0.6}\kern\dimen@#1}%
	    \fi
	  }%
	  \macc@depth\@ne
	  \let\math@bgroup\@empty \let\math@egroup\macc@set@skewchar
	  \mathsurround\z@ \frozen@everymath{\mathgroup\macc@group\relax}%
	  \macc@set@skewchar\relax
	  \let\mathaccentV\macc@nested@a
	  \if#31
	    \macc@nested@a\relax111{#1}%
	  \else
	    \def\gobble@till@marker##1\endmarker{}%
	    \futurelet\first@char\gobble@till@marker#1\endmarker
	    \ifcat\noexpand\first@char A\else
	      \def\first@char{}%
	    \fi
	    \macc@nested@a\relax111{\first@char}%
	  \fi
	  \endgroup
	}
	
	\def\@seccntformat#1{\@ifundefined{#1@cntformat}%
	   {\csname the#1\endcsname\quad}  % default
	   {\csname #1@cntformat\endcsname}% enable individual control
	}
	\let\oldappendix\appendix %% save current definition of \appendix
	\renewcommand\appendix{%
	    \oldappendix
	    \newcommand{\section@cntformat}{\appendixname$:\;$~\thesection}
	}
\makeatother

\newcommand{\eps}{\varepsilon}
\renewcommand{\Re}{\mathrm{Re}\,}
\newcommand{\mb}[1]{\mathbf{#1}}
\newcommand*{\ccdot}{\kern-.12em\cdot\kern-.12em}
\newcommand{\supp}{\mathrm{supp}}
\newcommand{\varrow}{\overset{v}{\longrightarrow}}

\newcommand{\convccirc}[1]{\mathbin{\mathop{\mathsmaller{\circledcirc}}\limits_{\raisebox{3.5pt}{\scalebox{0.75}{$#1$}}}}}
\newcommand\restrict[1]{\raisebox{-.5ex}{$|$}_{#1}}

\newtheoremstyle{slplain}% name
  {0.4cm}% Space above
  {0.4cm}% Space below
  {\upshape}% Body font
  {}%Indent amount (empty = no indent, \parindent = para indent)
  {\bfseries}%  Thm head font
  {.}%       Punctuation after thm head
  { }%      Space after thm head: " " = normal interword space;
  {}%       Thm head spec
  
\newtheoremstyle{itplain}% name
    {0.4cm}% Space above
    {0.4cm}% Space below
    {\itshape}% Body font
    {}%Indent amount (empty = no indent, \parindent = para indent)
    {\bfseries}%  Thm head font
    {.}%       Punctuation after thm head
    { }%      Space after thm head: " " = normal interword space;
    {}%       Thm head spec
  
\declaretheorem[style=slplain,numberwithin=section]{definition}
\declaretheorem[style=slplain,sibling=definition]{example}
\declaretheorem[style=slplain,sibling=definition]{remark}

\declaretheorem[style=itplain,sibling=definition]{theorem}

\declaretheorem[style=itplain,sibling=definition]{proposition}

\declaretheorem[style=itplain,sibling=definition]{corollary}

\AtBeginDocument{
	\addtolength{\abovedisplayskip}{0.125ex}
	\addtolength{\abovedisplayshortskip}{0.125ex}
	\addtolength{\belowdisplayskip}{0.125ex}
	\addtolength{\belowdisplayshortskip}{0.125ex}
}

\renewenvironment{abstract}{%
\noindent\hfill\begin{minipage}{0.92\textwidth}
\rule{\textwidth}{1pt}}
{\par\noindent\rule{\textwidth}{1pt}\end{minipage}\hfill}

\let\OLDthebibliography\thebibliography
\renewcommand\thebibliography[1]{
  \OLDthebibliography{#1}
  \setlength{\parskip}{0pt}
  \setlength{\itemsep}{3pt plus 0.3ex}
}

\title{\bfseries On the construction of convolution-like operators associated with multidimensional diffusion processes\\[0.3cm]}

\author{Rúben Sousa
\thanks{CMUP, Departamento de Matemática, Faculdade de Ciências, Universidade do Porto, Rua do Campo Alegre 687, 4169-007 Porto, Portugal. Email: \texttt{rubensousa@fc.up.pt}}
\and
Manuel Guerra \thanks{Corresponding author. ISEG – School of Economics and Management, Universidade de Lisboa; REM – Research in Economics and Mathematics, CEMAPRE, Rua do Quelhas 6, 1200-781 Lisbon, Portugal. Email: \texttt{mguerra@iseg.ulisboa.pt}}
\and
Semyon Yakubovich \thanks{CMUP, Departamento de Matemática, Faculdade de Ciências, Universidade do Porto, Rua do Campo Alegre 687, 4169-007 Porto, Portugal. Email: \texttt{syakubov@fc.up.pt}}
\\[0.3cm]
}
\date{\today\\}

\begin{document}

\maketitle

\begin{abstract}
	 \small
	 \parbox{\linewidth}{\vspace{-2pt}
\begin{center} \bfseries Abstract \vspace{-8pt} \end{center}
\-\ \quad 
When is it possible to interpret a given Markov process as a Lévy-like process?
Since the class of Lévy processes can be defined by  the relation between transition probabilities and convolutions, the answer to this question lies in the existence of a convolution-like operator satisfying the same relation with the transition probabilities of the process.
It is known that the so-called Sturm-Liouville convolutions have the desired properties and therefore the question above has a positive answer for a certain class of one-dimensional diffusions. 
However, more general processes have never been systematically treated in the literature.
	 	
\-\ \quad This study addresses this gap by considering the general problem of constructing a convolution-like operator for a given strong Feller process on a general locally compact metric space.
Both necessary and sufficient conditions for the existence of such convolution-like structures are determined, which reveal a connection between the answer to the above question and certain analytical and geometrical properties of the eigenfunctions of the transition semigroup.
	 	
The case of reflected Brownian motions on bounded domains of $\mathbb R^d$ and compact Riemannian manifolds is considered in greater detail:
various special cases are analysed, and a general discussion on the existence of appropriate convolution-like structures is presented.
	 
\vspace{5pt}

\-\ \quad \textbf{Keywords:}
Convolution-like operator, diffusion process, Lévy process, eigenfunction expansion, Neumann Laplacian, product formula. \vspace{6.5pt}
     }
\end{abstract}

\vspace{8pt}

\begingroup
\let\clearforchapter\relax
\section{Introduction}
\endgroup

We start by recalling a basic fact: \emph{if $\{X_t\}_{t \geq 0}$ is a Brownian motion on $\mathbb{R}^d$, then there exists a bilinear operator $\diamond$ on the space of probability measures on $\mathbb{R}^d$ such that the law of $X_t$ can be written as
\begin{equation} \label{eq:intro_brownianlevy}
P\bigl[X_t \in dy | X_s = x\bigr] = (\mu_{t-s} \diamond \delta_x)(dy), \qquad 0 \leq s \leq t,\; x \in \mathbb{R}^d
\end{equation}
where the measures $\mu_t$ satisfy the semigroup property $\mu_{t+s} = \mu_t \diamond \mu_s$}. The operator $\diamond$ is, of course, the ordinary convolution $(\mu \diamond \nu)(B) := \int_{\mathbb{R}^d} \int_{\mathbb{R}^d} \delta_{x+y}(B)\mu(dx) \nu(dy)$, and the identity \eqref{eq:intro_brownianlevy} is the convolution semigroup property of Brownian motion, i.e.\ it expresses the fact that the Brownian motion is a Lévy process. We can therefore interpret the class of Lévy processes as the natural generalization of the Brownian motion into a family of processes which share the same divisibility property. Remarkably, this family admits a complete characterization: this is the content of the famous Lévy-Khintchine theorem. 

Suppose now that $\{X_t\}_{t \geq 0}$ is a Feller process on a general locally compact metric space $E$. In light of the property stated above, one may ask: \emph{is it possible to construct an operator $\diamond$ on the space of probability measures on $E$ such that the law of $\{X_t\}$ admits the convolution-like semigroup representation \eqref{eq:intro_brownianlevy}?} 
An affirmative answer would show that the given process can be embedded into a family of processes characterized by a natural analogue of the notions of stationarity and independence.
In addition, it is reasonable to expect that such a convolution-like operator will have properties which allow us to develop analogues of the basic concepts and results in harmonic analysis. 
If so, and if the process' infinitesimal generator is an elliptic differential operator $\mathcal{L}$, then the convolution-like structure becomes a valuable tool for studying elliptic and parabolic differential equations determined by the operator $\mathcal{L}$.

The case of one-dimensional diffusion processes -- whose generators are Sturm-Liouville operators $\ell(u) := {1 \over r} \bigl( -(pu')' + qu \bigr)$ on intervals of the real line -- has been addressed by many authors \cite{achourtrimeche1979,bloomheyer1994,chebli1974,koornwinder1984,levitan1960,rentzschvoit2000,zeuner1992}. It turns out that the crucial ingredient for constructing an associated convolution-like operator is the existence of a family of probability measures $\bm{\nu}_{x,y}$ such that the product formula $w_\lambda(x) \, w_\lambda(y) = \int w_\lambda \, d\bm{\nu}_{x,y}$ holds for the solutions $w_\lambda$\, ($\lambda \geq 0$) of the Sturm-Liouville equation $\ell(u) = \lambda u$ with Neumann boundary condition. The existence of such product formulas for a large class of Sturm-Liouville operators has been established via a partial differential equation technique \cite{rentzschvoit2000,sousaetal2019b,sousaetal2020}. The resulting \emph{Sturm-Liouville convolutions} lead to a better understanding of the mapping properties of the eigenfunction expansion determined by the Sturm-Liouville operator, and they constitute a natural environment for studying infinite divisibility, Lévy-like processes and other notions from probabilistic harmonic analysis. In particular, the convolution-like semigroup representation \eqref{eq:intro_brownianlevy} holds for the corresponding one-dimensional diffusions.

The existing theory on convolution-like operators associated with stochastic processes is mostly limited to one-dimensional diffusions, but the question formulated above is meaningful for a much more general class of stochastic processes.
This paper provides a general discussion of the construction of convolutions for strong Feller processes on a locally compact separable metric space. 
Unlike previous studies, where applications to stochastic processes were regarded as a by-product of a pre-existing convolution-like structure \cite{borowieckaolszewska2015,heyerpap2010,rentzschvoit2000}, here the inclusion of a given process in some class of L\'evy-like processes is the main motivation.

This paper is organized as follows:
Section 2, contains the definition of Feller-L\'evy trivializable convolutions, i.e., the basic properties which an operator $\diamond$ should satisfy to fulfil the requirements outlined above.
It also contains some key examples of diffusions (namely Brownian motions on various spaces) which are within the scope of our general discussion, and an overview of some applications to harmonic analysis which can be derived from the existence of a Feller-L\'evy trivializable convolution.
In Section \ref{sec:characterizations} we derive both necessary and sufficient conditions relating the existence of the convolution structure with certain properties of the eigenfunctions of the generator of the Feller semigroup, such as existence of a common maximizer or positivity of a regularized product formula kernel.
Most of this analysis is focused on Feller semigroups on locally compact metric spaces whose spectrum is discrete.
In the subsequent Section \ref{sec:onedimensional} we digress into the (one-dimensional) case of semigroups generated by Sturm-Liouville operators so as to show that similar results hold without discreteness assumptions on the spectrum. Section \ref{sec:ultrahyperbolic} is devoted to the relation between convolution-like stuctures and (ultra)hyperbolic boundary value problems.

In Section \ref{sec:nonexist_multidim} we discuss the existence of common maximizers for the eigenfunctions of multidimensional diffusions. Some examples are presented which illustrate that one should not expect this property to hold unless the state space has a special geometric configuration. We then prove, using standard results on spectral theory of differential operators, that the common maximizer property does not hold for reflected Brownian motions on smooth domains of $\mathbb{R}^d$ ($d \geq 2$) or on compact Riemannian manifolds; this leads to a nonexistence theorem for convolutions on such domains. \clearpage

\section{Background} \label{sec:background}

Throughout the paper, the notations $\mathcal{M}_\mathbb{C}(E)$, $\mathcal{M}_+(E)$ and $\mathcal{P}(E)$ stand for, respectively, the spaces of finite complex, finite positive, and probability measures on a measurable space $E$. We write $\mu_n \varrow \mu$ to indicate that the sequence $\mu_n$ converges to $\mu$ in the vague topology of measures. The Dirac measure at a point $x$ is denoted by $\delta_x$, and the total variation norm of a measure $\mu \in \mathcal{M}_\mathbb{C}(E)$ is denoted by $\|\mu\|$. The ball centred at $x$ with radius $\eps > 0$ is denoted by $\mathds{B}(x;\eps)$. We denote by $\mathrm{C}(E)$ the space of continuous functions on $E$; $\mathrm{C}_\mathrm{b}(E)$ the subspace of bounded continuous functions; $\mathrm{C}_0(E)$ the subspace of functions vanishing at infinity; $\mathrm{C}_\mathrm{c}(E)$ the subspace of compactly supported functions; $\mathrm{C}^k(E)$ the space of functions which are $k$ times continuously differentiable; and $\mathrm{C}_\mathrm{c}^k(E)$ stands for the intersection $\mathrm{C}_\mathrm{c}(E) \cap \mathrm{C}^k(E)$. We also denote by $\mathrm{B}_\mathrm{b}(E)$ the space of bounded measurable (complex-valued) functions on $E$, provided with the supremum norm $\| f \|_\infty = \sup \{ |f(x)|: x \in E\}$. For $p \in [1,+\infty]$, $L^p(E,\mu)$ denotes the Lebesgue space of complex-valued functions on $E$, with the usual norm $\| \cdot \|_{L_p(E,\mu)}$.

\subsection{The notion of Feller-Lévy trivializable convolution}

Since the seminal works of Delsarte \cite{delsarte1938} and Levitan \cite{levitan1940}, many axiomatic convolution-like structures -- such as generalized convolutions, generalized translations, hypercomplex systems and hypergroups -- have been proposed with the aim of identifying the essential features which allow one to derive analogues of the basic facts of classical harmonic analysis \cite{berezansky1998,bloomheyer1994,litvinov1987,urbanik1964}.
In this paper we introduce and study the following notion of convolution-like structure:

\begin{definition} \label{def:fltc_fltcdef}
Let $E$ be a locally compact separable metric space, and let $\{T_t\}_{t \geq 0}$ be a strong Feller semigroup on $E$. We say that a bilinear operator $\diamond$ on $\mathcal{M}_\mathbb{C}(E)$ is a \emph{Feller-Lévy trivializable convolution} (FLTC) for $\{T_t\}$ if the following conditions hold:
\begin{enumerate}[itemsep=1.5pt]
\item[\textbf{I.}] $(\mathcal{M}_\mathbb{C}(E), \diamond)$ is a commutative Banach algebra over $\mathbb{C}$ (with respect to the total variation norm) with identity element $\delta_a$ ($a \in E$), and $(\mu,\nu) \mapsto \mu \diamond \nu$ is continuous in the weak topology of measures;
\item[\textbf{II.}] $\mathcal{P}(E) \diamond \mathcal{P}(E) \subset \mathcal{P}(E)$;
\item[\textbf{III.}] There exists a family $\Theta \subset \mathrm{B}_\mathrm{b}(E) \setminus \{0\}$ such that
\[ 
\mu = \mu_1 \diamond \mu_2 \qquad \text{if and only if} \qquad \mu(\vartheta) = \mu_1(\vartheta) \ccdot \mu_2(\vartheta) \; \text{for all } \vartheta \in \Theta ,
\]
where $\mu(\vartheta) := \int_E \vartheta(\xi) \mu(d\xi)$;
\item[\textbf{IV.}] The transition kernel $\{p_{t,x}\}_{t \geq 0, x \in E}$ of the semigroup $\{T_t\}$ is of the form
\begin{equation} \label{eq:fltc_convsemigrprop}
p_{t,x} = \gamma_t \diamond \delta_x,
\end{equation}
where $\{\gamma_t\}_{t \geq 0} \subset \mathcal{P}(E)$ is a family of measures such that $\gamma_{t+s} = \gamma_t \diamond \gamma_s$ for all $t, s \geq 0$.
\end{enumerate}
\end{definition}

Conditions I and II in the above definition can be seen as basic axioms allowing us to interpret $(\mathcal{M}_\mathbb{C}(E), \diamond)$ as a probability-preserving convolution-like structure.
Condition III requires the existence of an integral transform with bounded kernels which determines uniquely a given measure $\mu \in \mathcal{M}_\mathbb{C}(E)$ (in the sense that if $\mu(\vartheta) = \nu(\vartheta)$ for all $\vartheta \in \Theta$, then $\mu = \nu$) and trivializes the convolution in the same way as the Fourier transform trivializes the ordinary convolution.
As noted in \cite{volkovich1989}, it is possible, in principle, to study infinite divisibility of probability measures on measure algebras not satisfying Condition III;
however, it is natural to require Condition III to hold, not only because, to the best of our knowledge, all known examples of convolution-like structures are constructed from a product formula of the form $\vartheta(x) \vartheta(y) = (\delta_x \diamond \delta_y)(\vartheta)$ (and therefore possess such a trivializing family of functions) but also because this trivialization property leads to a richer theory.
Lastly, condition IV expresses the probabilistic motivation mentioned above:
the Feller semigroup $\{T_t\}$ is conservative and has the convolution semigroup property with respect to the operator $\diamond$;
in other words, a Feller process $\{X_t\}_{t \geq 0}$ associated to $\{T_t\}$ is a Lévy process with respect to $\diamond$, in the sense that we have $P\bigl[X_t \in \bm{\cdot} \mskip0.5\thinmuskip | X_s = x\bigr] = \gamma_{t-s} \diamond \delta_x$ for every $0 \leq s \leq t$ and $x \in E$.

\subsection{Feller semigroups on locally compact metric spaces: key examples}

The problem of existence of an associated FLTC is meaningful for any given strong Feller semigroup on a locally compact separable metric space.
In this subsection we introduce some basic examples of strong Feller semigroups which constitute model cases that should be kept in mind while examining the general results of Section \ref{sec:characterizations}.
This subsection also serves as a preparation for Section \ref{sec:nonexist_multidim}, where we will provide a definitive answer to whether it is possible to construct an FLTC for (reflected) Brownian motions on domains of $\mathbb{R}^d$ and Riemannian manifolds.

We begin by recalling some notions from the theory of Dirichlet forms.
Let $\mb{m}$ be a $\sigma$-finite measure on $E$. We say that $(\mathcal{E}, \mathcal{D}(\mathcal{E}))$ is a \emph{Dirichlet form} on $L^2(E,\mb{m})$ if $\mathcal{D}(\mathcal{E})$ is a dense subspace of $L^2(E,\mb{m})$ and $\mathcal{E}$ is a nonnegative, closed, Markovian symmetric sesquilinear form\label{not:sesquilinear2} defined on $\mathcal{D}(\mathcal{E}) \times \mathcal{D}(\mathcal{E})$. The associated non-positive self-adjoint operator $(\mathcal{G}^{(2)}, \mathcal{D}(\mathcal{G}^{(2)}))$ is defined as
\[
u \in \mathcal{D}(\mathcal{G}^{(2)}) \quad\; \text{ if and only if } \quad\; \exists \mskip0.5\thinmuskip \phi \in L^2(E,\mb{m}) \: \text{ such that }\: \mathcal{E}(u,v) = -\langle\phi, v\rangle_{L^2(E,\mb{m})} \text{ for all } v \in \mathcal{D}(\mathcal{E})
\]
and $\mathcal{G}^{(2)} u := \phi$ for $u \in \mathcal{D}(\mathcal{G}^{(2)})$. The semigroup determined by $\mathcal{E}$ is defined by $T_t^{(2)} := e^{t \mathcal{G}^{(2)}}$ (where the latter is obtained by spectral calculus); one can show \cite[Theorem 1.1.3]{chenfukushima2011} that $\{T_t^{(2)}\}$ is a strongly continuous, sub-Markovian contraction semigroup on $L^2(E,\mb{m})$. The Dirichlet form $(\mathcal{E}, \mathcal{D}(\mathcal{E}))$ is said to be \emph{strongly local} if $\mathcal{E}(u,v) = 0$ whenever $u \in \mathcal{D}(\mathcal{E})$ has compact support and $v \in \mathcal{D}(\mathcal{E})$ is constant on a neighbourhood of $\mathrm{supp}(u)$. It is said to be \emph{regular} if $\mathcal{D}(\mathcal{E}) \cap \mathrm{C}_\mathrm{c}(E)$ is dense both in $\mathcal{D}(\mathcal{E})$ with respect to the norm $\|u\|_{\mathcal{D}(\mathcal{E})} = \sqrt{\mathcal{E}(u,u) + \|u\|_{L^2(E,\mb{m})}}$ and in $\mathrm{C}_\mathrm{c}(E)$ with respect to the sup norm. A well-known result \cite[Theorem 7.2.1]{fukushimaetal2011} states that if $(\mathcal{E},\mathcal{D}(\mathcal{E}))$ is a regular Dirichlet form on $L^2(E,\mb{m})$ with semigroup $\{T_t^{(2)}\}_{t \geq 0}$, then there exists a Hunt process with state space $E$ whose transition semigroup $\{P_t\}_{t \geq 0}$ is such that $P_t u$ is, for all $u \in \mathrm{C}_\mathrm{c}(E)$, a quasi-continuous version of $T_t^{(2)} u$. (A \emph{Hunt process} is essentially a strong Markov process whose paths are right-continuous and quasi-left-continuous; for details we refer to \cite[Appendix A.2]{fukushimaetal2011}.) 

Now, we proceed with the announced examples.

\begin{example} \label{exam:fltc_feller_riemannian}
Let $(E,g)$ be a complete Riemannian manifold, let $\mb{m}$ be the Riemannian volume on $E$ and let $\nabla$ denote the Riemannian gradient on $(E,g)$. The sesquilinear form
\[
\mathcal{E}(u,v) = {1 \over 2} \int_E \langle \nabla u, \nabla v \rangle_g \, d\mb{m}, \qquad u, v \in \mathcal{D}(\mathcal{E})
\]
with domain
\[
\mathcal{D}(\mathcal{E}) = \text{closure of }  \mathrm{C}_\mathrm{c}^\infty(E) \text{ in the Sobolev space } H^1(E) \equiv \{u \in L^2(E,\mb{m}) \mid |\nabla u| \in L^2(E,\mb{m}) \}
\]
is a strongly local regular Dirichlet form on $L^2(E,\mb{m})$.
The Hunt process $\{X_t\}_{t \geq 0}$ with state space $E$ associated with this Dirichlet form is a \emph{Brownian motion on $(E,g)$}. One can show that the strongly continuous contraction semigroup $\{T_t\}$ determined by $\mathcal{E}$ is such that $T_t\bigl(\mathrm{B}_\mathrm{b}(E)\bigr) \subset \mathrm{C}_\mathrm{b}(E)$, so that the Brownian motion $\{X_t\}$ is a strong Feller process \cite[Section 6]{sturm1998a}. Moreover, it is shown in \cite[Example 5.7.2]{fukushimaetal2011} that the Feller semigroup $\{T_t\}$ is conservative provided that the Riemannian volume $\mb{m}$ is such that
\[
\liminf_{r \to \infty} {1 \over r^2} \log \mb{m}(\mathds{B}(x_0;r)) < \infty \text{ for some fixed } x_0 \in E.
\]
Let $\mathcal{G}: \mathcal{D}(\mathcal{G}) \subset \mathrm{C}_0(E) \longrightarrow \mathrm{C}_0(E)$ be the infinitesimal generator of the Brownian motion $\{X_t\}$. Then $\mathcal{G} u = {1 \over 2}\Delta u$ for $u \in \mathrm{C}_\mathrm{c}^\infty(E) \subset \mathcal{D}(\mathcal{G})$, where $\Delta$ is the Laplace-Beltrami operator on the Riemannian manifold $(E,g)$.
\end{example}

\begin{example} \label{exam:fltc_feller_euclidean}
Let $E = \mathbb{R}^d$,\, $m$ a positive function such that $m, {1 \over m} \in \mathrm{C}_\mathrm{b}(\mathbb{R}^d)$ and $A = (a_{jk})$ a symmetric $d \times d$ matrix-valued function such that $a_{jk} \in \mathrm{C}(\mathbb{R}^d)$ (for each $j,k \in \{1,\ldots,d\}$) and
\begin{equation} \label{eq:fltc_uniformelliptcond}
c^{-1} |\xi|^2 \leq \sum_{j,k=1}^d a_{jk}(x) \xi_j \xi_k \leq c |\xi|^2, \qquad (x,\xi) \in \mathbb{R}^d \times \mathbb{R}^d
\end{equation}
for some constant $c > 0$. The sesquilinear form
\[
\mathcal{E}(u,v) = {1 \over 2} \sum_{j,k=1}^d \int_{\mathbb{R}^d} a_{jk}(x) {\partial u \over \partial x_j} {\partial \overline{v\rule{0pt}{0.26\baselineskip}} \over \partial x_k} m(x) dx, \qquad u,v \in \mathcal{D}(\mathcal{E})
\]
with domain
\[
\mathcal{D}(\mathcal{E}) = \text{closure of } \mathrm{C}_\mathrm{c}^\infty(\mathbb{R}^d) \text{ under the inner product } \mathcal{E}(\bm{\cdot},\bm{\cdot}) + \langle \bm{\cdot},\bm{\cdot} \rangle_{L^2(\mb{m})}
\]
is a strongly local regular Dirichlet form on the space $L^2(\mb{m}) \equiv L^2(\mathbb{R}^d, m(x)dx)$ \cite[Section 3.1]{fukushimaetal2011}.
The Hunt process $\{X_t\}_{t \geq 0}$ associated with the Dirichlet form $\mathcal{E}$ is conservative \cite[Example 5.7.1]{fukushimaetal2011}.
The process $\{X_t\}$, which is called an \emph{$(A,m)$-diffusion} on $\mathbb{R}^d$, is a strong Feller process, cf.\ \cite[Example 4.C and Proposition 7.5]{sturm1998b}.
If, in addition, the functions ${\partial (m \mskip0.6\thinmuskip a_{jk}) \over \partial x_j} \mskip.7\thinmuskip$ ($j,k \in \{1,\ldots,d\}$) are locally square-integrable, then the infinitesimal generator $\mathcal{G}$ of the Feller semigroup is the elliptic operator $(\mathcal{G} u)(x) = {1 \over 2m(x)} \sum_{j,k=1}^d {\partial \over \partial x_j}\bigl( m(x) a_{jk}(x) {\partial u \over \partial x_k} \bigr)$\, (for $u \in \mathrm{C}_\mathrm{c}^2(\mathbb{R}^d) \subset \mathcal{D}(\mathcal{G})$).
\end{example}

\begin{example} \label{exam:fltc_feller_boundeddom}
Let $E$ be the closure of a bounded Lipschitz domain $\mathring{E} \subset \mathbb{R}^d$ and, as usual, let $H^k(E)$\, ($k \in \mathbb{N}$) be the Sobolev space\label{not:sobolev} defined as $H^k(E) := \{u \in L^2(E,dx) \mid \partial^\alpha u \in L^2(E,dx) \text{ for all } \alpha = (\alpha_1,\ldots,\alpha_d) \text{ with } |\alpha| \leq k\}$. Let $m \in H^1(E)$ be a positive function such that $m, {1 \over m} \in \mathrm{C}(E)$ and let $A = (a_{jk})$ be a symmetric bounded $d \times d$ matrix-valued function such that $a_{jk} \in H^1(E)$ for $j,k \in \{1,\ldots,d\}$ and the uniform ellipticity condition \eqref{eq:fltc_uniformelliptcond} holds for $(x,\xi) \in E \times \mathbb{R}^d$. The sesquilinear form
\[
\mathcal{E}(u,v) = {1 \over 2} \sum_{j,k=1}^d \int_{E} a_{jk}(x) {\partial u \over \partial x_j} {\partial \overline{v\rule{0pt}{0.26\baselineskip}} \over \partial x_k} m(x) dx, \qquad u,v \in \mathcal{D}(\mathcal{E}) = H^1(E)
\]
is a strongly local regular Dirichlet form on $L^2(E,\mb{m}) \equiv L^2(E,m(x)dx)$ whose associated Hunt process is a conservative Feller process, cf.\ \cite{chenzhang2014,chenfan2015}.
The process $\{X_t\}$ is called an \emph{$(A,m)$-reflected diffusion} on $E$. The infinitesimal generator $\mathcal{G}$ of the Feller process $\{X_t\}$ is such that $\mathrm{C}_\mathrm{c}^2(\mathring{E}) \subset \mathcal{D}(\mathcal{G})$ and $(\mathcal{G} u)(x) = {1 \over 2m(x)} \sum_{j,k=1}^d {\partial \over \partial x_j}\bigl( m(x) a_{jk}(x) {\partial u \over \partial x_k} \bigr)$ for $u \in \mathrm{C}_\mathrm{c}^2(\mathring{E})$. In the special case $a_{ij} = \delta_{ij}$ and $m = \mathds{1}$, the $(A,m)$-reflected diffusion is known as the \emph{reflected Brownian motion} on $E$, whose infinitesimal generator $\mathcal{G} u = {1 \over 2} \Delta u$ is the so-called \emph{Neumann Laplacian} on $E$.
\end{example}

\begin{example}
Let $E$ be a locally compact separable metric space with distance $\mb{d}$\label{not:distance2} and let $\mb{m}$ be a locally finite Borel measure on $E$ with $\mb{m}(U) > 0$ for all nonempty open sets $U \subset E$. 
Suppose that the triplet $(E,\mb{d},\mb{m})$ satisfies the \emph{measure contraction property} introduced in \cite[Definition 4.1]{sturm1998b}; roughly speaking, this means that there exists a family of quasi-geodesic maps connecting almost every pair of points $x,y \in E$ and which satisfy a contraction property which controls the distortions of the measure $\mb{m}$ along each quasi-geodesic. It was proved in \cite{sturm1998b} that the family of Dirichlet forms defined as
\[
\mathcal{E}^r(u,u) = {1 \over 2} \int_E \int_{\mathds{B}(x;r) \setminus \{x\}} \biggl|{u(z) - u(x) \over \mb{d}(z,x)}\biggr|^2 {\mb{m}(dz) \over \sqrt{\mb{m}(\mathds{B}(z;r))}} {\mb{m}(dx) \over \sqrt{\mb{m}(\mathds{B}(x;r))}}, \qquad r > 0
\]
(and $\mathcal{E}^r(u,v)$ defined via the polarization identity) converges as $r \downarrow 0$ (in a suitable variational sense, see \cite{sturm1998b}) to a strongly local regular Dirichlet form on $L^2(E,\mb{m})$. The associated Hunt process $\{X_t\}_{t \geq 0}$ is a strong Feller process with state space $E$.
If the growth of the volumes $\mb{m}(\mathds{B}(x;r))$ satisfies the condition stated in \cite[Theorem 5.7.3]{fukushimaetal2011}, then $\{X_t\}$ is conservative. This class of strong Feller processes includes, as particular cases, the diffusions of Examples \ref{exam:fltc_feller_riemannian} and \ref{exam:fltc_feller_euclidean} above, diffusions on manifolds with boundaries or corners, spaces obtained by gluing together manifolds, among others.
\end{example}

\subsection{Harmonic analysis with respect to Feller-Lévy trivializable convolutions}

The body of the paper is devoted to deriving necessary and sufficient conditions for the existence of an FLTC in a general framework that includes the examples given in the preceding subsection. As noted above, the motivation for this comes from the intention of extending the central concepts in the theory of harmonic analysis to a given Feller semigroup. Before proceeding to our main topic, we glance at some of the basic machinery that one can set up after having solved the problem of constructing an FLTC.

A Feller semigroup $\{T_t\}_{t \geq 0}$ on $E$ is said to be \emph{symmetric} with respect to a positive Borel measure $\mb{m}$ if $\int_E (T_t f)(x) \, g(x) \mskip0.5\thinmuskip \mb{m}(dx) = \int_E f(x) \, (T_t g)(x) \mskip0.5\thinmuskip \mb{m}(dx)$ for $f, g \in \mathrm{C}_\mathrm{c}(E)$. For simplicity, in this subsection we assume that the metric space $E$ is compact and the given strong Feller semigroup $\{T_t\}_{t \geq 0}$ is symmetric with respect to a finite measure $\mb{m} \in \mathcal{M}_+(E)$. Unless otherwise stated, we also assume that there exists an FLTC for the semigroup $\{T_t\}_{t \geq 0}$ whose trivializing family is $\Theta = \{\varphi_j\}_{j \in \mathbb{N}}$, where the $\varphi_j$ are real-valued eigenfunctions of the generator of $\{T_t\}$ and constitute a countable orthogonal basis of $L^2(E,\mb{m})$. (By the results in Section \ref{sec:characterizations}, we expect the trivializing family to be of this form whenever $E$ is compact.)

The following result ensures the existence of an invariant measure on the algebra $(\mathcal{M}_\mathbb{C}(E), \diamond)$ and summarizes the mapping properties of the convolution on the spaces $L^p(E,\mb{m})$.

\begin{proposition}
Let $\{T_t\}$ be a strong Feller semigroup on a compact space $E$, and let $\diamond$ be an FLTC for $\{T_t\}$. Let $(\mathcal{G}, \mathcal{D}(\mathcal{G}))$ be the generator of $\{T_t\}$.
\begin{enumerate}[itemsep=2pt]
\item[\textbf{(a)}] The identity $\delta_x \diamond \mb{m} = \mb{m}$ holds for all $x \in E$;
\item[\textbf{(b)}] Let $1 \leq p \leq \infty$ and $\mu \in \mathcal{M}_\mathbb{C}(E)$. If $f \in L^p(E,\mb{m})$, then the function $(\mathcal{T}^\mu f)(x) := \int_E f\, d(\mu \diamond \delta_x)$ is Borel measurable in $x \in E$ and satisfies $\|\mathcal{T}^\mu f\|_p \leq \|\mu\| \ccdot \|f\|_p$\, (where $\|\bm{\cdot}\|_p \equiv \|\bm{\cdot}\|_{L^p(E,\mb{m})}$).
\item[\textbf{(c)}] Let $p_1,p_2 \in [1, \infty]$ such that ${1 \over p_1} + {1 \over p_2} \geq 1$, and write $\mathcal{T}^x := \mathcal{T}^{\delta_x}$\, ($x \in E$). For $f \in L^{p_1}(E,\mb{m})$ and $h \in L^{p_2}(E,\mb{m})$, the $\diamond$-convolution 
\[
(f \diamond h)(x) = \int_M (\mathcal{T}^y f)(x)\, h(y)\, \mb{m}(dy)
\]
is well-defined and, for $r \in [1, \infty]$ defined by ${1 \over r} = {1 \over p_1} + {1 \over p_2} - 1$, it satisfies
\[
f \diamond h \in L^r(E,\mb{m}) \qquad \text{with} \qquad \| f \diamond h \|_r \leq \| f \|_{p_1} \| h \|_{p_2}.
\]
\item[\textbf{(d)}] The Banach space $L^1(E,\mb{m})$, equipped with the convolution multiplication $f \cdot h \equiv f \diamond h$, is a commutative Banach algebra without identity element.
\item[\textbf{(e)}] If $f \in \mathcal{D}(\mathcal{G})$ and $h \in L^1(E,\mb{m})$ then $\,f \diamond h \in \mathcal{D}(\mathcal{G})$ and $\:\mathcal{G} (f \diamond h) = (\mathcal{G} f) \diamond h$.
\end{enumerate}
\end{proposition}

\begin{proof}
By the assumptions above we have $(T_t \varphi_j)(x) = e^{-\lambda_j t} \varphi_j(x)$, where $\lambda_j$ is the eigenvalue of the generator corresponding to the eigenfunction $\varphi_j$ (see the proof of Proposition \ref{prop:fltc_trivTheta_L2basis} below). Using this, one can check that for all $j$ we have $(\gamma_t \diamond \delta_x \diamond \delta_y)(\varphi_j) = \int_E \varphi_j(\xi) \, q_t(x,y,\xi) \, \mb{m}(d\xi)$, where
\[
q_t(x,y,\xi) := \sum_{j=1}^\infty {1 \over \|\varphi_j\|_2^2} e^{-\lambda_j t} \varphi_j(x) \, \varphi_j(y) \, \varphi_j(\xi)
\]
and thus the measures $(\gamma_t \diamond \delta_x \diamond \delta_y)(d\xi)$ and $q_t(x,y,\xi) \, \mb{m}(d\xi)$ coincide. Consequently, for $f \in \mathrm{C}(E)$ we have
\begin{align*}
\int_E (\delta_x \diamond \delta_y)(f) \, \mb{m}(dy) & = \lim_{t \downarrow 0} \int_E f(\xi) q_t(x,y,\xi) \mb{m}(d\xi) \, \mb{m}(dy) \\
& = \lim_{t \downarrow 0} \int_E q_t(x,y,\xi) \mb{m}(dy) \, f(\xi) \mb{m}(d\xi) = \int_E f(y) \mb{m}(dy)
\end{align*}
which proved part (a).

Parts (b)--(e) are proved by arguing as in \cite[Section 4]{sousaetal2020} and \cite[Proposition 4.9]{sousaetal2019a}.
\end{proof}

The basic notions of the theory of infinite divisibility of probability measures can be readily extended to the measure algebra determined by an FLTC:

\begin{definition}
Let $\{T_t\}$ be a strong Feller semigroup on a compact space $E$, and let $\diamond$ be an FLTC for $\{T_t\}$.
\begin{itemize}
\item The set $\mathcal{P}_{\mathrm{id}}$ of \emph{$\diamond$-infinitely divisible measures} is defined by
\[
\mathcal{P}_{\mathrm{id}} = \bigl\{ \mu \in \mathcal{P}(E) \bigm| \text{for all } n \in \mathbb{N} \text{ there exists } \nu_n \in \mathcal{P}(E) \text{ such that } \mu = (\nu_n)^{\mskip -0.5\thinmuskip \diamond n} \bigr\}
\]
where $(\nu_n)^{\mskip -0.5\thinmuskip \diamond n}$ denotes the $n$-fold $\diamond$-convolution of the measure $\nu_n$ with itself.

\item The \emph{$\diamond$-Poisson measure}\label{not:poisson_conic} associated with $\nu \in \mathcal{M}_+(E)$ is
\[
\mb{e}(\nu): = e^{-\|\nu\|} \sum_{n=0}^\infty {\nu^{\mskip -0.5\thinmuskip \diamond n} \over n!}
\]
(the infinite sum converging in total variation).

\item A measure $\mu \in \mathcal{P}(E)$ is called a \emph{$\diamond$-Gaussian measure} if $\mu \in \mathcal{P}_{\mathrm{id}}$ and
\[
\mu = \mb{e}(\nu) \diamond \vartheta \quad \bigl( \nu \in \mathcal{M}_+(E),\, \vartheta \in \mathcal{P}_{\mathrm{id}}\bigr) \qquad \implies \qquad \mb{e}(\nu) = \delta_a.
\]
\end{itemize}
\end{definition}

Like in the case of the ordinary convolution, there exists a one-to-one correspondence between $\diamond$-infinitely divisible measures and \emph{$\diamond$-convolution semigroups}, i.e.\ families $\{\mu_t\}_{t \geq 0} \subset \mathcal{P}(E)$ such that
\begin{equation} \label{eq:intro_convsemigr}
\mu_{t+s} = \mu_t \diamond \mu_s \;\;\text{ for all } t, s \geq 0, \qquad\;\; \mu_0 = \delta_a, \qquad\;\; t \longmapsto \mu_t \text{ is weakly continuous}.
\end{equation}
Indeed, we can state: \emph{if $\{\mu_t\}$ is a $\diamond$-convolution semigroup, then $\mu_t$ is (for each $t \geq 0$) a $\diamond$-infinitely divisible distribution; conversely, if $\mu$ is a $\diamond$-infinitely divisible distribution, then the semigroup $\{\mu_t\}$ defined by $\mu_t(f) = \mu^{\diamond t}(f)$ is the unique $\diamond$-convolution semigroup such that $\mu_1 = \mu$.} (The proof relies on an argument similar to that for the ordinary convolution, see \cite[Section 7]{sato1999}, \cite[Section 5.3]{bloomheyer1994}.)

The following Lévy-Khintchine type representation for infinitely divisible measures is proved in \cite{volkovich1988} (see also \cite[Subsection 10.2]{jewett1975}).

\begin{theorem} \label{thm:intro_levykhin}
Let $\{T_t\}$ be a strong Feller semigroup on a compact space $E$, and let $\diamond$ be a FLTC for $\{T_t\}$. Let $\mu \in \mathcal{P}_\mathrm{id}$ be a measure such that 
\begin{equation} \label{eq:intro_levykhin_hyp}
\delta_x \diamond \mu \neq \mu \quad \text{ for all } x \in E \setminus \{a\}.
\end{equation}
Then the trivializing integrals $\mu(\varphi_j)$ can be represented in the form
\begin{equation} \label{eq:intro_levykhin}
\mu(\varphi_j) = \alpha(\varphi_j) \exp\biggl( \int_{E \setminus \{a\}\!} \bigl( \varphi_j(x) - 1 \bigr) \nu(dx) \biggr) \qquad (j \in \mathbb{N})
\end{equation}
where $\nu$ is a $\sigma$-finite measure on $E \setminus \{a\}$ which is finite on the complement of any neighbourhood of $a$ and such that
\[
\int_{E \setminus \{a\}\!} \bigl( 1 - \varphi_j(x) \bigr) \nu(dx) < \infty \qquad (j \in \mathbb{N})
\]
and $\alpha$ is a $\diamond$-Gaussian measure.

Conversely, if $\nu$ is a $\sigma$-finite measure satisfying the stated conditions and $\alpha$ is a $\diamond$-Gaussian measure then there exists $\mu \in \mathcal{P}_\mathrm{id}$ satisfying \eqref{eq:intro_levykhin_hyp} and such that \eqref{eq:intro_levykhin} holds for all $j \in \mathbb{N}$.
\end{theorem}

The next proposition shows that the problem of constructing Feller semigroups associated with a given convolution measure algebra -- which is the converse of the problem studied in this paper -- has a straightforward solution. (The proof is straightforward and does not depend on the existence of an associated trivializing family, cf.\ e.g.\ \cite[Proposition 2.1]{rentzschvoit2000}.)

\begin{proposition} \label{prop:intro_converseprob}
Let $E$ be a compact metric space and let $\diamond$ be a bilinear operator on $\mathcal{M}_\mathbb{C}(E)$ satisfying conditions I and II of Definition \ref{def:fltc_fltcdef}. Let $\{\mu_t\} \subset \mathcal{P}(E)$ be a $\diamond$-convolution semigroup, i.e.\ a family of measures for which \eqref{eq:intro_convsemigr} holds. For $\nu \in \mathcal{M}^+(E)$, let $\mathcal{T}^\nu: \mathrm{C}(E) \longrightarrow \mathrm{C}(E)$ be the operator defined as
\[
(\mathcal{T}^\nu f)(x) := \int_E f\, d(\nu \diamond \delta_x).
\]
Then the family of operators $\{\mathcal{T}^{\mu_t}\}_{t \geq 0}$ constitutes a conservative Feller semigroup.
\end{proposition}

An $E$-valued time-homogeneous Markov process $\{X_t\}_{t \geq 0}$ is called a \emph{$\diamond$-Lévy process} if its transition probabilities $p_{t,x} = P[X_t \in \bm{\cdot} | X_0 = x]$ can be written in the form \eqref{eq:fltc_convsemigrprop} for some $\diamond$-convolution semigroup $\{\mu_t\}$. Some equivalent martingale characterizations are given next. (The proof is the same as in \cite[Theorem 3.4]{rentzschvoit2000}.)

\begin{proposition} 
Let $\{T_t\}$ be a strong Feller semigroup on a compact space $E$, and let $\diamond$ be an FLTC for $\{T_t\}$. Let $\{\mu_t\}_{t \geq 0}$ be a $\diamond$-convolution semigroup and $\psi_j := \log\mu_1(\varphi_j)$. Let $(\mathcal{A}, \mathcal{D}(\mathcal{A}))$ be the generator of the Feller semigroup $\{\mathcal{T}^{\mu_t}\}_{t \geq 0}$ (Proposition \ref{prop:intro_converseprob}), and let $X$ be an $E$-valued càdlàg Markov process. The following assertions are equivalent:
\begin{enumerate}[itemsep=0pt,topsep=4pt]
\item[\textbf{(i)}] $X$ is a $\diamond$-Lévy process associated with $\{\mu_t\}$;
\item[\textbf{(ii)}] $\{e^{t \mskip0.5\thinmuskip \psi_j} \varphi_j(X_t)\}_{t \geq 0}$ is a martingale for each $j \in \mathbb{N}$;
\item[\textbf{(iii)}] $\bigl\{ \varphi_j(X_t) - \varphi_j(X_0) + \psi_j \int_0^t \varphi_j(X_s)\, ds \bigr\}_{t \geq 0}$ is a martingale for each $\lambda \geq 0$;
\item[\textbf{(iv)}] $\{f(X_t) - f(X_0) - \int_0^t (\mathcal{A}f)(X_s)\, ds\}$ is a martingale for each $f \in \mathcal{D}(\mathcal{A})$.
\end{enumerate}
\end{proposition}

We refer to \cite{berezansky1998,bloomheyer1994} and references therein for other properties of convolution algebras determined by FLTCs which constitute analogues of standard results in harmonic analysis.

\section{Characterizations of Feller-Lévy trivializable convolutions on locally compact spaces} \label{sec:characterizations}

\subsection{A necessary condition: the common maximizer property}

In this subsection we will prove that the FLTC axioms (i.e.\ conditions I--IV of Definition \ref{def:fltc_fltcdef}) entail a strong restriction on the behaviour of the eigenfunctions of the Feller semigroup: their maximum must be located at a common point of the metric space $E$.

Given a Feller semigroup $\{T_t\}_{t \geq 0}$ on $E$, its \emph{$\eta$-resolvent operator} is defined as
\[
\mathcal{R}_\eta: \mathrm{C}_\mathrm{b}(E) \longrightarrow \mathrm{C}_\mathrm{b}(E), \qquad\quad \mathcal{R}_\eta f := \int_0^\infty e^{-\eta t} \mskip0.8\thinmuskip T_t f \, dt \qquad\quad (\eta > 0)
\]
and we define the \emph{$\mathrm{C}_\mathrm{b}$-generator} $\bigl(\mathcal{G}^{(b)}, \mathcal{D}(\mathcal{G}^{(b)})\bigr)$ as the operator with domain $\mathcal{D}(\mathcal{G}^{(b)}) = \mathcal{R}_\eta\bigl(\mathrm{C}_\mathrm{b}(E)\bigr)$ and given by
\[
(\mathcal{G}^{(b)} u)(x) = \eta u(x) - g(x) \qquad \text{ for }\, u = \mathcal{R}_\eta g, \;\; g \in \mathrm{C}_\mathrm{b}(E), \;\; x \in E.
\]
(One can check that $\mathcal{G}^{(b)}$ is independent of $\eta$.) The following proposition and corollary show that the existence of an FLTC for $\{T_t\}$ implies that all the trivializing functions are eigenfunctions of the $\mathrm{C}_\mathrm{b}$-generator and, in addition, are uniformly bounded by their value at the identity element $a \in E$.

\begin{proposition} \label{prop:fltc_prodform_eigenf}
Let $\{T_t\}$ be a strong Feller semigroup on a locally compact separable metric space $E$, and let $\diamond$ be a bilinear operator on $\mathcal{M}_\mathbb{C}(E)$ satisfying conditions I, II and IV of Definition \ref{def:fltc_fltcdef}. Suppose that $\vartheta \in \mathrm{B}_\mathrm{b}(E)$, $\vartheta \not\equiv 0$ is a function such that 
\begin{equation} \label{eq:fltc_prodform_functeq}
(\delta_x \diamond \delta_y)(\vartheta) = \vartheta(x) \ccdot \vartheta(y) \qquad \text{for all } x,y \in E.
\end{equation}
Then $\vartheta(a) = \|\vartheta\|_\infty = 1$. Moreover, $\vartheta$ is an eigenfunction of the $\mathrm{C}_\mathrm{b}$-generator $(\mathcal{G}^{(b)}, \mathcal{D}(\mathcal{G}^{(b)}))$ associated with an eigenvalue of nonpositive real part, in the sense that we have $\vartheta \in \mathcal{D}(\mathcal{G}^{(b)})$ and $\mathcal{G}^{(b)} \vartheta = -\lambda \vartheta$ for some $\lambda \in \mathbb{C}$ with $\Re \lambda \geq 0$.
\end{proposition}

\begin{proof}
Clearly, $\vartheta(x) = \delta_x(\vartheta) = (\delta_x \diamond \delta_a)(\vartheta) = \vartheta(x) \vartheta(a)$ for all $x \in E$. Since $\vartheta \not\equiv 0$, this implies that $\vartheta(a) = 1$.

Next, pick $\eps > 0$ and choose $x_0 \in E$ such that $|\vartheta(x_0)| > \| \vartheta \|_\infty - \eps$. Then $(\| \vartheta \|_\infty - \eps)^2 < |\vartheta(x_0)|^2 = |(\delta_{x_0} \diamond \delta_{x_0})(\vartheta)| \leq \|\vartheta\|_\infty$ (by condition II, $\delta_{x_0} \diamond \delta_{x_0} \in \mathcal{P}(E)$, which justifies the last step). Since $\eps$ is arbitrary, $\| \vartheta \|_\infty^2 \leq \|\vartheta\|_\infty$, hence $\|\vartheta\|_\infty \leq 1$.

Since $\diamond$ is bilinear and weakly continuous, a straightforward argument yields that $(\mu \diamond \nu)(d\xi) = \int_E \int_E (\delta_x \diamond \delta_y)(d\xi) \mu(dx) \nu(dy)$ for $\mu, \nu \in \mathcal{M}_\mathbb{C}(E)$. Consequently, \eqref{eq:fltc_prodform_functeq} implies that $(\mu \diamond \nu)(\vartheta) = \mu(\vartheta) \ccdot \nu(\vartheta)$ for all $\mu, \nu \in \mathcal{M}_\mathbb{C}(E)$. In particular,
\[
(T_t \vartheta)(x) \equiv p_{t,x}(\vartheta) = (\mu_t \diamond \delta_x)(\vartheta) = \mu_t(\vartheta) \ccdot \vartheta(x)
\]
due to condition IV. Given that $\{T_t\}$ is strong Feller, we have $T_t \vartheta \in \mathrm{C}_\mathrm{b}(E)$ and therefore $\vartheta = {T_t \vartheta \over \mu_t(\vartheta)} \in \mathrm{C}_\mathrm{b}(E)$.
Moreover, the fact that $\{T_t\}$ is a Feller semigroup ensures that $\mu_t(\vartheta) = (T_t \vartheta)(a)$ is a continuous function of $t$ which, by condition IV, satisfies the functional equation $\mu_{t+s}(\vartheta) = \mu_t(\vartheta) \mu_s(\vartheta)$. Therefore $\mu_t(\vartheta) = e^{-\lambda t}$ for some $\lambda \in \mathbb{C}$, and the fact that $|\mu_t(\vartheta)| \leq \|\vartheta\|_\infty = 1$ implies that $\Re\lambda \geq 0$. We thus have $T_t \vartheta = e^{-\lambda t} \vartheta$ and $\mathcal{R}_\eta \vartheta = {\vartheta \over \lambda + \eta}$ for $\eta > 0$, so we conclude that $\vartheta \in \mathcal{D}(\mathcal{G}^{(b)})$ and $\mathcal{G}^{(b)} \vartheta = -\lambda \vartheta$.
\end{proof}

It is worth noting that if the strong Feller semigroup $\{T_t\}$ is symmetric with respect to a finite measure $\mb{m} \in \mathcal{M}_+(E)$, then the space $\mathrm{C}_\mathrm{b}(E)$ is contained in $L^2(E,\mb{m})$; accordingly, the Feller semigroup $\{T_t\}_{t \geq 0}$ and the $\mathrm{C}_\mathrm{b}$-generator $\mathcal{G}^{(b)}$ extend, respectively, to a strongly continuous semigroup $\{T_t^{(2)}\}$ of symmetric operators on $L^2(E,\mb{m})$ and to the corresponding infinitesimal generator $\mathcal{G}^{(2)}$. In this setting, the trivializing functions $\vartheta \in \Theta$ are eigenfunctions of the $L^2$-generator $\mathcal{G}^{(2)}$. Applying spectral-theoretic results for self-adjoint operators on Hilbert spaces, we can deduce further properties for the trivializing functions:

\begin{proposition} \label{prop:fltc_trivTheta_L2basis}
Let $\{T_t\}$ be a Feller semigroup on a locally compact separable metric space $E$.

\smallskip
\noindent
\textbf{(a)} Suppose the corresponding transition kernels $\{p_{t,x}(\bm{\cdot})\}_{t > 0, x \in E}$ are of the form $p_{t,x}(dy) = p_t(x,y)\mb{m}(dy)$ for some finite measure $\mb{m} \in \mathcal{M}_+(E)$ and some density function $p_t(x,y)$ which is bounded and symmetric on $E \times E$ for each $t > 0$. Then:
\begin{enumerate}[leftmargin=3em]
\item[\textbf{(a1)}] $\{T_t\}$ is strong Feller, symmetric with respect to $\mb{m}$, and admits an extension $\{T_t^{(2)}\}$ which is a strongly continuous semigroup on the space $L^2(E,\mb{m})$;
\item[\textbf{(a2)}] There exists a sequence $0 \leq \lambda_1 \leq \lambda_2 \leq \lambda_3 \leq \ldots$, and an orthonormal basis $\{\omega_j\}_{j \in \mathbb{N}}$ of $L^2(E,\mb{m})$ such that
\[
T_t^{(2)} \omega_j = e^{-\lambda_j t} \omega_j \quad (t \geq 0), \qquad\; \mathcal{G}^{(2)} \omega_j = -\lambda_j \omega_j
\]
where $\mathcal{G}^{(2)}$ stands for the generator of the $L^2$-semigroup $\{T_t^{(2)}\}$. The sequence of eigenvalues is such that $\sum_{j=1}^\infty e^{-\lambda_j t} < \infty$ for each $t > 0$ (so that, in particular, $\lim_{j \to \infty} \lambda_j = \infty$).
\end{enumerate}
\smallskip
\noindent
\textbf{(b)} Assume also that $\diamond$ is an FLTC for $\{T_t\}$ and that $\Theta$ is the family of trivializing functions for $\diamond$. Write $S_k = \{j \in \mathbb{N}\mid \lambda_j = \lambda_k\}$ and $m_k = |S_k|$\, ($k \in \mathbb{N}$). Then each function $\vartheta \in \Theta$ is a solution of $\mathcal{G}^{(2)}\vartheta = -\lambda_j \vartheta$ for some $j \in \mathbb{N}$. Furthermore, there exist functions $\{\vartheta_j\}_{j \in \mathbb{N}} \subset \Theta$ such that
\[
\mathrm{span}(\{\omega_j\}_{j \in S_k}) = \mathrm{span}(\{\vartheta_j\}_{j \in S_k})
\]
and, consequently, $\overline{\mathrm{span}}(\Theta) = L^2(E,\mb{m})$.
\end{proposition}

\begin{proof}
\textbf{(a1)} The strong Feller property follows from \cite[Theorem 1.14]{bottcher2013}. The symmetry with respect to $\mb{m}$ is obvious, and it is straightforward to show that for $f \in \mathrm{C}_\mathrm{c}(E)$ we have $\|T_t f\|_{L^2(E,\mb{m})} \leq \|f\|_{L^2(E,\mb{m})}$ and $\|T_t f - f\|_{L^2(E,\mb{m})} \to 0$ as $t \downarrow 0$, so that the extension $\{T_t^{(2)}\}$ is a strongly continuous semigroup on $L^2(E,\mb{m})$.
\smallskip

\textbf{(a2)} Let $\langle \bm{\cdot},\bm{\cdot}\rangle$ be the inner product on $L^2(E,\mb{m})$. By the spectral theorem for compact self-adjoint operators (cf.\ e.g.\ \cite[Theorem 6.7]{teschl2014}), the operator $T_1^{(2)}$ can be written as $T_1^{(2)} = \sum_{j=1}^\infty \mu_j \langle\omega_j,\bm{\cdot}\rangle \, \omega_j$, where $\mu_1 \geq \mu_2 \geq \ldots$ are the eigenvalues of $T_1^{(2)}$ and $\{\omega_j\}_{j \in \mathbb{N}}$ is an orthonormal basis of $L^2(E,\mb{m})$ such that each $\omega_j$ is an eigenfunction of $T_1^{(2)}$ associated with the eigenvalue $\mu_j$; in addition, we have $\mu_1 \leq \|T_1^{(2)}\|$ and $\mu_j \downarrow 0$ as $j \to \infty$. If we define $\lambda_j = -\log\mu_j$, then it follows that $T_t^{(2)} = \sum_{j=1}^\infty e^{-\lambda_j t} \langle\omega_j,\bm{\cdot}\rangle \, \omega_j$. (This can be justified as follows, cf.\ \cite[pp.\ 463--464]{getoor1959} for further details: we know that $(T_1^{(2)} - \mu_j)\omega_j = (T_{1/2}^{(2)} + \mu_j^{1/2})(T_{1/2}^{(2)} - \mu_j^{1/2}) \omega_j = 0$, and all the eigenvalues of $(T_{1/2}^{(2)} + \mu_j^{1/2})$ are positive, hence $T_{1/2}^{(2)} \omega_j = \mu_j^{1/2} \omega_j$; similarly $T_t^{(2)} \omega_j = e^{-\lambda_j t} \omega_j$ for all $t = m/2^k$ and thus, by strong continuity, for all $t > 0$.) Consequently, $\mathcal{G}^{(2)} \omega_j = \lim_{t \downarrow 0} \tfrac{1}{t} (T_t^{(2)} \omega_j - \omega_j) =  -\lambda_j \omega_j$. Since $\mb{m}$ is a finite measure and the densities $p_t(\bm{\cdot},\bm{\cdot})$ are bounded, the operator $T_t^{(2)}$ is, for each $t > 0$, a Hilbert-Schmidt operator, and therefore we have $\sum_{j=1}^\infty e^{-\lambda_j t} < \infty$ for all $t > 0$.
\smallskip

\textbf{(b)} By Proposition \ref{prop:fltc_prodform_eigenf}, each $\vartheta \in \Theta$ is such that $\mathcal{G}^{(2)} \vartheta = -\lambda \vartheta$ for some $\lambda \in \mathbb{C}$. Given that $\Theta \subset L^2(E,\mb{m})$ and eigenfunctions associated with different eigenvalues are orthogonal, we have $\lambda = \lambda_j$ because otherwise we get a contradiction with the basis property of $\{\omega_j \}$.
Next, fix $t>0$, $k \in \mathbb{N}$ and let $\Theta_k := \{\vartheta \in \Theta \mid T_t^{(2)} \vartheta = e^{-\lambda_k t} \vartheta\} \subset L^2(E,\mb{m})$. Given that $\{\omega_j\}_{j \in S_k}$ is a basis of the eigenspace associated with $\lambda_k$, we have $\mathrm{span}(\Theta_k) \subset \mathrm{span}(\{\omega_j\}_{j \in S_k})$. To prove the reverse inclusion, let $h \in \mathrm{span}(\{\omega_j\}_{j \in S_k}) \cap \mathrm{span}(\Theta_k)^\bot$, write $\nu_h(dx) := h(x) \mb{m}(dx)$ and observe that (since $\mb{m}$ is a finite measure) $\nu_h \in \mathcal{M}_\mathbb{C}(E)$. Then the integral
\[
\nu_h(\vartheta) = \int_E \vartheta(x) h(x) \mb{m}(dx)
\]
is equal to zero for $\vartheta \in \Theta_k$ because $h \in \mathrm{span}(\Theta_k)^\bot$, and is also equal to zero for $\vartheta \in \Theta \setminus \Theta_k$ because then $h$ and $\vartheta$ are eigenfunctions of $T_t^{(2)}$ associated with different eigenvalues. Since measures $\nu \in \mathcal{M}_\mathbb{C}(E)$ are uniquely determined by the integrals $\{\nu(\vartheta)\}_{\vartheta \in \Theta}$, it follows that $\nu_h = 0$ and therefore $h = 0$; this shows that $\mathrm{span}(\Theta_k) = \mathrm{span}(\{\omega_j\}_{j \in S_k})$. It follows at once that there exist linearly independent functions $\{\vartheta_j\}_{j \in \mathbb{N}} \subset \Theta$ such that $\mathrm{span}(\{\omega_j\}_{j \in S_k}) = \mathrm{span}(\{\vartheta_j\}_{j \in S_k})$.
\end{proof}

The conclusions of Proposition \ref{prop:fltc_trivTheta_L2basis} are valid, in particular, for the Feller semigroups associated with the Brownian motion on a compact Riemannian manifold or with an $(A,m)$-reflected diffusion on a bounded Lipschitz domain, cf.\ Examples \ref{exam:fltc_feller_riemannian} and \ref{exam:fltc_feller_boundeddom} respectively. (Indeed, it follows from e.g.\ \cite[Theorem 7.4]{sturm1998b} that in both cases we have $p_{t,x}(dy) = p_t(x,y) \mb{m}(dy)$ with $p_t(x,y)$ bounded and symmetric; recall also that compact Riemannian manifolds have finite volume, cf.\ e.g.\ \cite[Theorem 3.11]{grigoryan2009}.)

The bottom line of Propositions \ref{prop:fltc_prodform_eigenf} and \ref{prop:fltc_trivTheta_L2basis} is that if there exists an FLTC for a Feller semigroup $\{T_t\}$ satisfying the assumptions above, then the following \emph{common maximizer property} holds:
\begin{enumerate}
\item[\textbf{CM.}] There exists a set $\{\vartheta_j\}_{j \in \mathbb{N}}$ of eigenfunctions of $\mathcal{G}^{(2)}$ and a point $a \in E$ such that $\mathrm{span}\{\vartheta_j \}_{j \in \mathbb N}$ is dense in $L^2(E, \mathbf m)$ and  $\vartheta_j (a)  = \|\vartheta_j \|_\infty = 1$ for every $j \in \mathbb N$.
\end{enumerate}
(The functions $\vartheta_j$ associated with a common eigenvalue need not be orthogonal in $L^2(E,\mb{m})$.) The common maximizer property will play a fundamental role in the proof of the inexistence results established in Subsection \ref{sub:eigenexp_critical_nonex}.

\subsection{Sufficient conditions on compact metric spaces}

In the context of Feller semigroups determined by one-dimensional diffusions, it has been shown that the positivity of the kernel of a regularized product formula is a key ingredient for the construction of probability-preserving convolution-like operators \cite[Section 4]{sousaetal2019b}. The next result shows that a similar positivity condition is enough to ensure the existence of an FLTC for a Feller semigroup on a general compact metric space.

\begin{proposition} \label{prop:fltc_prodform_kerncharact}
Under the assumptions of part (a) of Proposition \ref{prop:fltc_trivTheta_L2basis}, assume that the metric space $E$ is compact.
Let $0 \leq \lambda_1 \leq \lambda_2 \leq \lambda_3 \leq \ldots$ be the eigenvalues of $-\mathcal{G}^{(2)}$ and let $\{\varphi_j\}_{j \in \mathbb{N}} \subset L^2(E,\mb{m})$ be an orthogonal set of functions such that
\[
T_t^{(2)} \varphi_j = e^{-\lambda_j t} \varphi_j, \qquad \varphi_1 = \mathds{1}, \qquad \sup_j \|\varphi_j\|_2 < \infty
\]
where $\|\bm{\cdot}\|_2$ denotes the norm of the space $L^2(E,\mb{m})$. Then $\varphi_j \in \mathrm{C}(E)$ for all $j \in \mathbb{N}$, and the series $\sum_{j=1}^\infty {1 \over \|\varphi_j\|_2^2} e^{-\lambda_j t} \varphi_j(x) \, \varphi_j(y) \, \varphi_j(\xi)$ is absolutely convergent for all $t > 0$ and $x,y,\xi \in E$. Moreover, the following assertions are equivalent:
\begin{enumerate}
\item[\textbf{(i)}] We have 
\begin{equation} \label{eq:fltc_timeshiftkern_positiv}
q_t(x,y,\xi) := \sum_{j=1}^\infty {1 \over \|\varphi_j\|_2^2} e^{-\lambda_j t} \varphi_j(x) \, \varphi_j(y) \, \varphi_j(\xi) \geq 0
\end{equation}
for all $t > 0$ and $x, y, \xi \in E$.
\item[\textbf{(ii)}] For each $x, y \in E$ there exists a measure $\bm{\nu}_{x,y} \in \mathcal{P}(E)$ such that the product $\varphi_j(x) \, \varphi_j(y)$ admits the integral representation
\begin{equation} \label{eq:fltc_tshypprodform}
\varphi_j(x) \, \varphi_j(y) = \int_E \varphi_j(\xi)\, \bm{\nu}_{x,y}(d\xi), \qquad x, y \in E, \; j \in \mathbb{N}.
\end{equation}
\end{enumerate}
If these equivalent conditions hold and, in addition, there exists $a \in E$ such that $\varphi_j(a) = 1$ for all $j \in \mathbb{N}$, then the bilinear operator $\diamond$ on $\mathcal{M}_\mathbb{C}(E)$ defined as
\begin{equation} \label{eq:fltc_kerncharact_convdef}
(\mu \diamond \nu)(d\xi) = \int_E \int_E \bm{\nu}_{x,y}(d\xi) \, \mu(dx) \, \nu(dy)
\end{equation}
is an FLTC for $\{T_t\}$ with trivializing family $\Theta = \{\varphi_j\}_{j \in \mathbb{N}}$.
\end{proposition}

\begin{proof}
Denote the inner product of the space $L^2(E,\mb{m})$ by $\langle \bm{\cdot}, \bm{\cdot }\rangle$. For each $\eps > 0$ we have
\begin{equation} \label{eq:fltc_prodform_kerncharact_pf1}
|\varphi_j(x)| = e^{\lambda_j \eps} |(T_\eps^{(2)} \varphi_j)(x)| = e^{\lambda_j \eps} \langle \varphi_j, p_\eps(x,\bm{\cdot}) \rangle \leq c_\eps \, \mb{m}(E)^{1/2}\, e^{\lambda_j \eps} \|\varphi_j\|_2 < \infty \qquad \text{for } \mb{m}\text{-a.e. } x \in E
\end{equation}
where $c_\eps = \sup_{(x,y) \in E \times E\,} p_\eps(x,y)$. This shows that the function $\varphi_j$ belongs to the space $\mathrm{B}_\mathrm{b}(E)$ (possibly after redefining $\varphi_j$ on a $\mb{m}$-null set). Since $\{T_t\}$ is strong Feller (Proposition \ref{prop:fltc_trivTheta_L2basis}), it follows that $\varphi_j = e^{\lambda_j \eps\mskip0.7\thinmuskip } T_\eps\varphi_j \in \mathrm{C}(E)$. The assumption that $\sup_j \|\varphi_j\|_2 < \infty$, together with the estimate \eqref{eq:fltc_prodform_kerncharact_pf1}, ensures that the series $\sum_{j=1}^\infty {1 \over \|\varphi_j\|_2^2} e^{-\lambda_j t} \varphi_j(x) \, \varphi_j(y) \, \varphi_j(\xi)$ is absolutely convergent.

Suppose that \eqref{eq:fltc_timeshiftkern_positiv} holds and fix $x, y \in E$. For $t > 0$, let $\bm{\nu}_{t,x,y} \in \mathcal{M}_+(E)$ be the measure defined by $\bm{\nu}_{t,x,y}(d\xi) = q_t(x,y,\xi) \mb{m}(d\xi)$. We have
\begin{equation} \label{eq:fltc_tshypprodform_pf1}
\begin{aligned}
\int_E \varphi_j(\xi) \, \bm{\nu}_{t,x,y}(d\xi) & = \int_E \varphi_j(\xi)\sum_{k=1}^\infty {1 \over \|\varphi_j\|_2^2} e^{-\lambda_k t} \varphi_k(x) \, \varphi_k(y) \, \varphi_k(\xi) \, \mb{m}(d\xi) \\
& = \sum_{k=1}^\infty {1 \over \|\varphi_j\|_2^2} e^{-\lambda_k t} \varphi_k(x) \, \varphi_k(y) \, \langle\varphi_j, \varphi_k\rangle \\
& = e^{-\lambda_j t} \varphi_j(x) \, \varphi_j(y)
\end{aligned}
\end{equation}
It then follows from \eqref{eq:fltc_tshypprodform_pf1} (with $j=1$) that $\bm{\nu}_{t,x,y}(E) = 1$, so that
\[
\bm{\nu}_{t,x,y} \in \mathcal{P}(E) \qquad \text{for all } t > 0,\; x, y \in E.
\]
Now, let $\{t_n\}_{n \in \mathbb{N}}$ be an arbitrary decreasing sequence with $t_n \downarrow 0$. Since any uniformly bounded sequence of finite positive measures contains a vaguely convergent subsequence, there exists a subsequence $\{t_{n_k}\}$ and a measure $\bm{\nu}_{x,y} \in \mathcal{M}_+(E)$ such that $\bm{\nu}_{t_{n_k},x,y} \varrow \bm{\nu}_{x,y}$ as $k \to \infty$. Let us show that all such subsequences $\{\bm{\nu}_{t_{n_k},x,y}\}$ have the same vague limit. Suppose that $t_k^{1}$, $t_k^{2}$ are two different sequences with $t_k^s \downarrow 0$ and that $\bm{\nu}_{t_k^s,x,y}\! \varrow \bm{\nu}_{x,y}^s$ as $k \to \infty$ ($s=1,2$). Recalling that $E$ is compact, it follows that for all $h \in \mathrm{C}(E)$ and $\eps > 0$ we have
\begin{align*}
\int_E (T_\eps h)(\xi) \, \bm{\nu}_{x,y}^s(d\xi) & = \lim_{k \to \infty} \int_E (T_\eps h)(\xi) \, \bm{\nu}_{t_k^s,x,y}(d\xi) \\
& = \lim_{k \to \infty} \sum_{j=1}^\infty {1 \over \|\varphi_j\|_2^2} e^{-\lambda_j(t_k^s + \eps)} \varphi_j(x) \, \varphi_j(y) \, \langle h, \varphi_j \rangle \\[-1.5pt]
& = \sum_{j=1}^\infty {1 \over \|\varphi_j\|_2^2} e^{-\lambda_j \eps} \varphi_j(x) \, \varphi_j(y) \, \langle h, \varphi_j \rangle
\end{align*}
where the second equality follows from the identities $\langle q_t(x,y,\bm{\cdot}),\varphi_j \rangle = e^{-\lambda_j t} \varphi_j(x) \, \varphi_j(y)$ and $\langle T_\eps h,\varphi_j \rangle = \langle h, T_\eps \varphi_j \rangle = e^{-\lambda_j \eps} \langle h, \varphi_j \rangle$. Consequently, we have
\begin{equation} \label{eq:fltc_tshypprodform_pf2}
\int_E (T_\eps h)(\xi) \, \bm{\nu}_{x,y}^1(d\xi) = \int_E (T_\eps h)(\xi) \, \bm{\nu}_{x,y}^2(d\xi) \qquad \text{for all } \eps > 0.
\end{equation}
Since $h \in \mathrm{C}(E)$, by strong continuity of the Feller semigroup $\{T_t\}$ we have $\lim_{\eps \downarrow 0} \| T_\eps h - h \|_\infty = 0$, so by taking the limit $\eps \downarrow 0$ in both sides of \eqref{eq:fltc_tshypprodform_pf2} we deduce that $\bm{\nu}_{x,y}^1(h) = \bm{\nu}_{x,y}^2(h)$, where $h \in \mathrm{C}(E)$ is arbitrary; therefore, $\bm{\nu}_{x,y}^1 = \bm{\nu}_{x,y}^2$. Thus all subsequences have the same vague limit, and from this we conclude that $\bm{\nu}_{t,x,y} \varrow \bm{\nu}_{x,y}$ as $t \downarrow 0$. The product formula \eqref{eq:fltc_tshypprodform} is then obtained by taking the limit $t \downarrow 0$ in the leftmost and rightmost sides of \eqref{eq:fltc_tshypprodform_pf1}.

Conversely, suppose that \eqref{eq:fltc_tshypprodform} holds for some measure $\bm{\nu}_{x,y} \in \mathcal{M}_+(E)$. Noting that for $h \in \mathrm{C}(E)$ we have \vspace{-5pt}
\begin{align*}
\bigl\langle h, p_t(x,\bm{\cdot}) \bigr\rangle = (T_t h)(x) & = \sum_{j=1}^\infty {1 \over \|\varphi_j\|_2^2} \langle T_t h, \varphi_j \rangle \, \varphi_j(x) \\
& = \sum_{j=1}^\infty {1 \over \|\varphi_j\|_2^2} e^{-\lambda_j t} \langle h, \varphi_j \rangle \, \varphi_j(x) \\[-1pt]
& = \Bigl\langle h,\sum_{j=1}^\infty {1 \over \|\varphi_j\|_2^2} e^{-\lambda_j t} \varphi_j(x) \varphi_j(\bm{\cdot}) \Bigr\rangle
\end{align*}
we see that $p_t(x,y) = \sum_{j=1}^\infty {1 \over \|\varphi_j\|_2^2} e^{-\lambda_j t} \varphi_j(x) \, \varphi_j(y)$. Consequently, we have
\[
q_t(x,y,\xi) = \sum_{j=1}^\infty {1 \over \|\varphi_j\|_2^2} e^{-\lambda_j t} \varphi_j(x) \! \int_E \varphi_j(z)\, \bm{\nu}_{y,\xi}(dz) = \int_E p_t(x,z)\, \bm{\nu}_{y,\xi}(dz) \geq 0 \qquad (t > 0,\; x,y \in E)
\]
because both the density $p_t(x,\bm{\cdot})$ and the measures $\bm{\nu}_{y,\xi}$ are nonnegative.

Finally, assume that $\varphi_j(a) = 1$ for all $j$ and that (ii) holds. Let $\diamond$ be the operator defined by \eqref{eq:fltc_kerncharact_convdef}. To prove that $\Theta = \{\varphi_j\}_{j \in \mathbb{N}}$ satisfies condition III in Definition \ref{def:fltc_fltcdef}, it only remains to show that each $\mu \in \mathcal{M}_\mathbb{C}(E)$ is uniquely characterized by $\{\mu(\varphi_j)\}_{j \in \mathbb{N}}$. Indeed, if we take $\mu \in \mathcal{M}_\mathbb{C}(E)$ such that $\mu(\varphi_j) = 0$ for all $j$, then for $h \in \mathrm{C}(E)$ and $t > 0$ we have
\[
\int_E (T_t h)(x) \, \mu(dx) = \int_E \, \sum_{j=1}^\infty {1 \over \|\varphi_j\|_2^2} e^{-\lambda_j t} \langle h, \varphi_j \rangle \, \varphi_j(x) \, \mu(dx) = \sum_{j=1}^\infty e^{-\lambda_j t} \langle h, \varphi_j \rangle \, \mu(\varphi_j)  = 0
\]
and this implies that $\mu(h) = 0$ for all $h \in \mathrm{C}(E)$, so that $\mu \equiv 0$. Using the fact that $\Theta$ satisfies condition III, we can easily check that $\diamond$ is commutative, associative, bilinear and has identity element $\delta_a$. It is also straightforward that $\|\mu \diamond \nu\| \leq \|\mu\| \ccdot \|\nu\|$ and that $\mathcal{P}(E) \diamond \mathcal{P}(E) \subset \mathcal{P}(E)$. If $x_n \to x$ and $y_n \to y$, then 
\[
(\delta_{x_n} \diamond \delta_{y_n})(\varphi_j) = \varphi_j(x_n) \varphi_j(y_n) \longrightarrow \varphi_j(x) \varphi_j(y) = (\delta_x \diamond \delta_y)(\varphi_j) \qquad (j \in \mathbb{N})
\]
and therefore (by a standard vague convergence argument, cf.\ \cite[Remark 6.4]{sousaetal2020}) $\delta_{x_n} \diamond \delta_{y_n} \varrow \delta_x \diamond \delta_y$; it is then straightforward to conclude that $(\mu, \nu) \mapsto \mu \diamond \nu$ is continuous in the weak topology. Noting that $p_{t,x}(\varphi_j) = e^{-\lambda_j t} \varphi_j(x) = p_{t,a}(\varphi_j) \delta_x(\varphi_j)$, we conclude that $\diamond$ is an FLTC for $\{T_t\}$.
\end{proof}

Combining the proposition above with the common maximizer property of the previous subsection, we obtain the following characterization for the existence of an FLTC associated with a Feller semigroup whose eigenvalues are simple:

\begin{corollary} \label{cor:fltc_prodform_simplspectr_iff}
In the conditions of Proposition \ref{prop:fltc_prodform_kerncharact}, assume that the operator $T_1^{(2)}$ has simple spectrum (i.e.\ all the eigenvalues $e^{-\lambda_j}$ have multiplicity $1$). Let $\{(\lambda_j,\omega_j)\}_{j \in \mathbb{N}}$ be the eigenvalue-eigenfunction pairs defined in Proposition \ref{prop:fltc_trivTheta_L2basis}. Then the following are equivalent:
\begin{enumerate}
\item[\textbf{(i)}] There exists an FLTC for $\{T_t\}_{t \geq 0}$;
\item[\textbf{(ii)}] There exists $a \in E$ such that $|\omega_j(a)| = \|\omega_j\|_\infty$ for all $j \in \mathbb{N}$, and the positivity condition \eqref{eq:fltc_timeshiftkern_positiv} holds for the eigenfunctions $\varphi_j(x) := {\omega_j(x) \over \omega_j(a)}$. 
\end{enumerate}
\end{corollary}

\begin{proof}
The implication (ii) $\!\implies\!$ (i) follows from the final statement in Proposition \ref{prop:fltc_prodform_kerncharact}. Conversely, if (i) holds then the common maximizer property discussed above implies that $\Theta = \{\varphi_j\}_{j \in \mathbb{N}}$ where $\varphi_j(x) := {\omega_j(x) \over \omega_j(a)}$; from this it follows (by condition III of Definition \ref{def:fltc_fltcdef}) that \eqref{eq:fltc_tshypprodform} holds with $\bm{\nu}_{x,y} = \delta_x \diamond \delta_y$ and therefore (by Proposition \ref{prop:fltc_prodform_kerncharact}) the $\varphi_j$ satisfy the positivity condition \eqref{eq:fltc_timeshiftkern_positiv}.
\end{proof}

We note here that the assumption that $T_t^{(2)}$ (or, equivalently, the generator $\mathcal{G}^{(2)}$) has no eigenvalues with multiplicity greater than $1$ is known to hold for many strong Feller semigroups of interest. In fact, it is proved in \cite[Example 6.4]{henry2005} that the property that \emph{all the eigenvalues of the Neumann Laplacian are simple} is a generic property in the set of all bounded connected $\mathrm{C}^2$ domains $E \subset \mathbb{R}^d$. (The meaning of this is the following: given a bounded connected $\mathrm{C}^2$ domain $E$, consider the collection of domains $\mathfrak{M}_3(E) = \{h(E) \mid h:E \longrightarrow \mathbb{R}^d \text{ is a } \mathrm{C}^3\text{-diffeomorphism}\}$, which is a separable Banach space, see \cite{henry2005} for details concerning the appropriate topology. Let $\mathfrak{M}_{\mathrm{simp}} \subset \mathfrak{M}_3(E)$ be the subspace of all $\widetilde{E} \in \mathfrak{M}_3(E)$ such that all the eigenvalues of the Neumann Laplacian on $\widetilde{E}$ are simple. Then $\mathfrak{M}_{\mathrm{simp}}$ can be written as a countable intersection of open dense subsets of $\mathfrak{M}_3(E)$.) Similar results hold for the Laplace-Beltrami operator on a compact Riemannian manifold: it was proved in \cite{uhlenbeck1976} that, given a compact manifold $M$, the set of Riemannian metrics $g$ for which all the eigenvalues of the Laplace-Beltrami operator on $(M,g)$ are simple is a generic subset of the space of Riemannian metrics on $M$.

However, one should not expect the property of simplicity of spectrum to hold for Euclidean domains or Riemannian manifolds with symmetries. For instance, if a bounded domain $E \subset \mathbb{R}^2$ is invariant under the natural action of the dihedral group $\bm{D}_n$, then one can show (see \cite{helfferetal2002}) that the Dirichlet or Neumann Laplacian on $E$ has infinitely many eigenvalues with multiplicity $\geq 2$.

\section{The one-dimensional case} \label{sec:onedimensional}

As mentioned in the Introduction, the construction of probability-preserving convolution-like operators for one-dimensional diffusion semigroups generated by Sturm-Liouville operators has been the subject of many papers. As one would expect, the known existence theorems rely on a detailed study of the Sturm-Liouville eigenfunction expansion. We refer to \cite{berezansky1998,bloomheyer1994,sousaetal2019b,sousaetal2020} and references therein for background on convolutions of Sturm-Liouville type.

The aim of this section is to prove that the necessary and sufficient condition for existence of FLTCs stated in Corollary \ref{cor:fltc_prodform_simplspectr_iff} can be extended to one-dimensional (Feller) diffusion semigroups on possibly unbounded intervals and whose spectrum needs not be discrete. This will build upon and generalize our previous results in \cite[Section 4]{sousaetal2019b}.

We first review some notions from Sturm-Liouville theory. Consider the Sturm-Liouville operator
\begin{equation} \label{eq:onedim_SLop}
\ell = -{1 \over r} {d \over dx} \Bigl( p \, {d \over dx}\Bigr), \qquad x \in (a,b)
\end{equation}
where we assume that the coefficients are such that $p(x), r(x) > 0$ for all $x \in (a,b)$,\, $p, p', r, r' \in \mathrm{AC_{loc}}(a,b)$, the endpoint $a$ is regular or entrance, and the endpoint $b$ is regular, entrance or natural. (The classification as a regular, entrance, exit or natural endpoint refers to the Feller boundary classification, cf.\ e.g.\ \cite[Section 5.11]{ito2006}.) Denote by $w_\lambda(\bm{\cdot})$ the solution of the initial value problem
\[
\ell(w) = \lambda w \quad (a < x < b, \; \lambda \in \mathbb{C}), \qquad\; w(a) = 1, \qquad\; (pw')(a) = 0.
\]
(The existence of a unique solution is proved in \cite[Lemma 2.1]{sousaetal2019b}.) The Neumann realization of $\ell$ is the operator $(\mathcal{L}^{(2)}, \mathcal{D}(\mathcal{L}^{(2)}))$ defined as
\[
\mathcal{D}(\mathcal{L}^{(2)}) := \begin{cases}
\bigl\{ u \in L^2(r) \bigm| u, u' \in \mathrm{AC}_{\mathrm{loc}}(a,b), \; \ell(u) \in L^2(r), \; (pu')(a) = 0 \bigr\} \qquad & \text{ if } b \text{ is natural}\\[2pt]
\biggl\{ u \in L^2(r) \biggm| \!\!
\begin{array}{l}
u, u' \in \mathrm{AC}_{\mathrm{loc}}(a,b), \; \ell(u) \in L^2(r), \\[-3pt]
(pu')(a) = (pu')(b) = 0
\end{array} \! \biggr\} \qquad & \hspace{-.114\linewidth}\text{ if } b \text{ is regular or entrance}\\
\end{cases}
\]
and $\mathcal{L}^{(2)} u = \ell(u)$ for $u \in \mathcal{D}(\mathcal{L}^{(2)})$. Let $\mathcal{F}: L^2(r) \longrightarrow L^2(\mathbb{R}; \bm{\rho}_\mathcal{L})$ be the eigenfunction expansion of the Neumann realization of $\ell$, which is an isometric isomorphism given by \cite[Proposition 2.5]{sousaetal2019b}
\begin{equation} \label{eq:onedim_eigenexp}
(\mathcal{F} h)(\lambda) := \int_a^b h(x) \, w_\lambda(x) \, r(x) dx, \qquad (\mathcal{F}^{-1} \varphi)(x) = \int_{\bm{\Lambda}} \varphi(\lambda) \, w_\lambda(x) \, \bm{\rho}_\mathcal{L}(d\lambda)
\end{equation}
where $\bm{\rho}_\mathcal{L}$ is a locally finite positive Borel measure on $\mathbb{R}$ and $\bm{\Lambda} := \supp(\bm{\rho}_\mathcal{L})$. Let $I = [a,b)$ if $b$ is natural and $I = [a,b]$ if $b$ is regular or entrance, and denote by $\{T_t^{(2)}\}_{t \geq 0}$ the semigroup generated by $(\mathcal{L}^{(2)}, \mathcal{D}(\mathcal{L}^{(2)}))$. By \cite[Proposition 2.7]{sousaetal2019b}, the restriction of $\{T_t^{(2)}\}$ to $\mathrm{C}_\mathrm{c}(a,b)$ extends to a Feller process $\{T_t\}_{t \geq 0}$ on $I$ which for $h \in \mathrm{C}_0(I) \cap L^2(r)$ admits the representations
\begin{equation} \label{eq:onedim_fellergen_tpdf}
(T_t h)(x) = \int_a^b h(y) \, p(t,x,y) \, r(y) dy = \int_{\bm{\Lambda}} e^{-t\lambda} w_\lambda(x) \, (\mathcal{F} h)(\lambda)\, \bm{\rho}_\mathcal{L}(d\lambda) \qquad t > 0, \, x \in (a,b)
\end{equation}
where $p(t,x,y) = \int_{\bm{\Lambda}} e^{-t\lambda} \, w_\lambda(x) \, w_\lambda(y)\, \bm{\rho}_\mathcal{L}(d\lambda)$.

We are now able to state the advertised characterization for the existence of an FLTC for the Feller semigroup generated by the Neumann realization of $\ell$:

\begin{proposition} \label{prop:fltc_prodform_kerncharact_onedim}
Assume that $e^{-t\mskip0.5\thinmuskip \bm{\pmb\cdot}} \! \in L^2(\bm{\Lambda};\bm{\rho}_\mathcal{L})$ for all $t > 0$. Then the following are equivalent:
\begin{enumerate}
\item[\textbf{(i)}] There exists an FLTC for $\{T_t\}_{t \geq 0}$ with trivializing family $\Theta = \{w_\lambda\}_{\lambda \in \bm{\Lambda}}$.
\item[\textbf{(ii)}] We have $w_\lambda \in \mathrm{C}_\mathrm{b}(I)$ for all $\lambda \in \bm{\Lambda}$, and the function
\begin{equation} \label{eq:fltc_onedim_regkernel}
q_t(x,y,\xi) := \int_{\bm{\Lambda}\!} e^{-t\lambda\,} w_\lambda(x) \, w_\lambda(y) \, w_\lambda(\xi) \, \bm{\rho}_{\mathcal{L}}(d\lambda) \qquad \bigl(t > 0, \; x, y, \xi \in (a,b)\bigr)
\end{equation}
is well-defined as an absolutely convergent integral; moreover, the measures defined as $\bm{\nu}_{t,x,y}(d\xi) = q_t(x,y,\xi) \, r(\xi) d\xi$ are such that $\{\bm{\nu}_{t,x,y}\}_{0 < t \leq 1,\, x, y \in (a,\beta]}$ is, for each $\beta < b$, a tight family of probability measures on $I$.
\end{enumerate}
\end{proposition}

\begin{proof}
\textbf{\emph{(i)$\!\implies\!$(ii):\;}} Let $\diamond$ be an FLTC for $\{T_t\}_{t \geq 0}$ and $p_{t,x} = \mu_t \diamond \delta_x$ the transition kernel of $\{T_t\}_{t \geq 0}$. Since the $w_\lambda$ are multiplicative linear functionals on the Banach algebra $(\mathcal{M}_\mathbb{C}(I), \diamond)$, we have $\|w_\lambda\|_\infty = 1$ for all $\lambda \in \bm{\Lambda}$, hence the right-hand side of \eqref{eq:fltc_onedim_regkernel} is absolutely convergent. Moreover,
\[
(\mu_t \diamond \delta_x \diamond \delta_y)(w_\lambda) = e^{-t\lambda\,} w_\lambda(x) \, w_\lambda(y) = \mathcal{F}[q_t(x,y,\bm{\cdot})](\lambda) \qquad \bigl(t > 0, \; x, y \in (a,b)\bigr)
\]
and it follows that for $g \in \mathrm{C}_\mathrm{c}(I)$ 
\begin{align*}
\int_I g(\xi) \, (\mu_t \diamond \delta_x \diamond \delta_y)(d\xi) & = \lim_{s \downarrow 0} \int_I (T_s g)(\xi) \, (\mu_t \diamond \delta_x \diamond \delta_y)(d\xi) \\
& = \lim_{s \downarrow 0} \int_{\bm{\Lambda}} (\mathcal{F}g)(\lambda) \, e^{-(t+s)\lambda\, } w_\lambda(x) \, w_\lambda(y) \, \bm{\rho}(d\lambda) \\
& = \int_I g(\xi) \, q_t(x,y,\xi) \, r(\xi) d\xi
\end{align*}
where we used \eqref{eq:onedim_fellergen_tpdf}, Fubini's theorem and the isometric property of $\mathcal{F}$. Since $g$ is arbitrary, this shows that the measures $(\mu_t \diamond \delta_x \diamond \delta_y)(d\xi)$ and $\bm{\nu}_{t,x,y}(d\xi) := q_t(x,y,\xi) \, r(\xi) d\xi$ coincide. Consequently, $\bm{\nu}_{t,x,y} \in \mathcal{P}(I)$ for all $t > 0$ and $x, y \in (a,b)$. Since $\{T_t\}$ is a Feller process, the mapping $(t,x) \mapsto p_{t,x} = \mu_t \diamond \delta_x$ is continuous on $[0,\infty) \times I$ with respect to the weak topology of measures, and therefore the family $\{\bm{\nu}_{t,x,y}\}_{0 < t \leq 1,\, x, y \in (a,\beta]}$ is relatively compact, hence tight. \\[-12pt]

\textbf{\emph{(ii)$\!\implies\!$(i):\;}} In the case where $b$ is regular or entrance, this implication follows from Proposition \ref{prop:fltc_prodform_kerncharact}. Assume that (ii) holds and that $b$ is natural. It follows from \eqref{eq:onedim_eigenexp} that
\begin{equation} \label{eq:fltc_prodform_kerncharact_onedim_pf2}
e^{-t\lambda} \, w_\lambda(x) \, w_\lambda(y) = \int_I w_\lambda(\xi)\, \bm{\nu}_{t,x,y}(d\xi) \qquad \bigl(t > 0, \; x, y \in (a,b), \; \lambda \in \bm{\Lambda}\bigr)
\end{equation}
where the integral converges absolutely. Since $\{\bm{\nu}_{t,x,y}\}_{0 < t \leq 1,\, x, y \in (a,\beta]}$ is tight, given $x,y \in I$ there exists a sequence $\{(t_n,x_n,y_n)\}_{n \in \mathbb{N}} \subset (0,+\infty) \times (a,b) \times (a,b)$ such that $(t_n,x_n,y_n) \to (0,x,y)$ and the measures $\bm{\nu}_{t_n,x_n,y_n}$ converge weakly to a measure $\bm{\nu}_{x,y} \in \mathcal{P}(I)$ as $n \to \infty$. Moreover, if $(x,y) \neq (a,a)$ and $\bm{\nu}_{x,y}^1, \bm{\nu}_{x,y}^2$ denote two such limits with approximating sequences $\{(t_n^j,x_n^j,y_n^j)\}_{n \in \mathbb{N}}$, then for $j = 1,2$ and $g \in \mathcal{D}(\mathcal{L}^{(2)}) \cap \mathrm{C}_\mathrm{c}(I)$ we have
\[
\int_I g(\xi) \, \bm{\nu}_{x,y}^{j}(d\xi) = \lim_{n \to \infty} \int_{\bm{\Lambda}} e^{-t_n^{j}\lambda\,} w_\lambda(x_n^{j}) \, w_\lambda(y_n^{j}) \, (\mathcal{F}g)(\lambda) \, \bm{\rho}_{\mathcal{L}}(d\lambda) = \int_{\bm{\Lambda}} w_\lambda(x) \, w_\lambda(y) \, (\mathcal{F}g)(\lambda) \, \bm{\rho}_{\mathcal{L}}(d\lambda)
\]
so that $\bm{\nu}_{x,y}^1 = \bm{\nu}_{x,y}^2$. This shows that the measures $\bm{\nu}_{t,\widetilde{x},\widetilde{y}}$ converge weakly as $(t,\widetilde{x},\widetilde{y}) \to (0,x,y)$ to a unique limit $\bm{\nu}_{x,y}$ which is characterized by the identity $\int_I g(\xi) \, \bm{\nu}_{x,y}(d\xi) = \int_{\bm{\Lambda}} w_\lambda(x)\, w_\lambda(y)\, (\mathcal{F}g)(\lambda)\, \bm{\rho}_\mathcal{L}(d\lambda)$\, ($g \in \mathcal{D}^{(2,0)}$, $x,y \in I$, $(x,y) \neq (a,a)$). Using this fact and the reasoning in the proof of \cite[Proposition 5.2(ii)]{sousaetal2019b}, we can verify that each measure $\mu \in \mathcal{M}_\mathbb{C}(I)$ is uniquely determined by the family of integrals $\{\mu(w_\lambda)\}_{\lambda \in \bm{\Lambda}}$. From this it follows, by taking limits in both sides of \eqref{eq:fltc_prodform_kerncharact_onedim_pf2}, that the measure $\bm{\nu}_{a,a} = \delta_a$ is the unique weak limit of $\bm{\nu}_{t_n,x_n,y_n}$ as $(t_n,x_n,y_n) \to (0,a,a)$.

For $\mu, \nu \in \mathcal{M}_\mathbb{C}(I)$, define $(\mu \diamond \nu)(d\xi) := \int_I \int_I \bm{\nu}_{x,y}(d\xi) \, \mu(dx) \, \nu(dy)$, where $\bm{\nu}_{x,y} \in \mathcal{P}(I)$ is the unique weak limit described above. Then \eqref{eq:fltc_prodform_kerncharact_onedim_pf2} yields that
\begin{equation} \label{eq:fltc_prodform_kerncharact_onedim_pf4}
w_\lambda(x) \, w_\lambda(y) = \int_I w_\lambda(\xi)\, (\delta_x \diamond \delta_y)(d\xi) \qquad \bigl(x, y \in I, \; \lambda \in \bm{\Lambda}\bigr)
\end{equation}
and, consequently, condition III of Definition \ref{def:fltc_fltcdef} holds with $\Theta = \{w_\lambda\}_{\lambda \in \bm{\Lambda}}$. It is also clear that $(\mathcal{M}_\mathbb{C}(I), \diamond)$ is a commutative Banach algebra over $\mathbb{C}$ with identity element $\delta_a$ and that $\mathcal{P}(I) \diamond \mathcal{P}(I) \subset \mathcal{P}(I)$. From the tightness of $\{\bm{\nu}_{t,x,y}\}_{0 < t \leq 1,\, x, y \in (a,\beta]}$\, ($\beta < b$) it follows that the family of limits $\{\bm{\nu}_{x,y}\}_{x, y \in (a,\beta]}$ is also tight; it then follows from \eqref{eq:fltc_prodform_kerncharact_onedim_pf4} that the map $(x,y) \mapsto \bm{\nu}_{x,y}$ is continuous with respect to the weak topology and, consequently, $(\mu,\nu) \mapsto \mu \diamond \nu$ is also continuous. Finally, it follows from \eqref{eq:onedim_eigenexp}--\eqref{eq:onedim_fellergen_tpdf} that the transition kernel $p_{t,x}(dy) \equiv p(t,x,y) \, r(y) dy$ is such that $p_{t,x}(w_\lambda) = p_{t,a}(w_\lambda) w_\lambda(x)$ for all $t > 0$ $x \in I$ and $\lambda \in \bm{\Lambda}$. We thus have $p_{t,x} = p_{t,a} \diamond \delta_x$, hence condition IV of Definition \ref{def:fltc_fltcdef} holds.
\end{proof}

We note that the assumption that $e^{-t\mskip0.5\thinmuskip \bm{\pmb\cdot}} \! \in L^2(\bm{\Lambda};\bm{\rho}_\mathcal{L})$ holds for all Sturm-Liouville operators whose left endpoint $a$ is regular, and it also holds for a fairly large class of operators for which the endpoint $a$ is entrance \cite{kotani1975,kotani2007}.

\section{Connection with hyperbolic and ultrahyperbolic equations} \label{sec:ultrahyperbolic}

In this section we briefly highlight the role of hyperbolic and ultrahyperbolic partial differential equations in the construction of an FLTC for a given Feller semigroup.

Consider first the one-dimensional case of Feller semigroups generated by Sturm-Liouville operators. By Proposition \ref{prop:fltc_prodform_kerncharact_onedim}, in order to prove the existence of an FLTC we need to ensure that
\[
\mathcal{Q}_{t,h}(x,y) := \int_{\bm{\Lambda}} e^{-t\lambda\,} w_\lambda(x) \, w_\lambda(y) \, (\mathcal{F}h)(\lambda) \, \bm{\rho}_{\mathcal{L}}(d\lambda) \geq 0 \qquad \text{for all }\, h \in \mathrm{C}_\mathrm{c}(I) \,\text{ and }\, t > 0.
\]
The function $\mathcal{Q}_{t,h}(x,y)$ is a solution of $\ell_x u = \ell_y u$ (where $\ell_x$ is the Sturm-Liouville operator \eqref{eq:onedim_SLop} acting on the variable $x$) satisfying the boundary conditions $\mathcal{Q}_{t,h}(x,a) = (T_t h)(x)$ and $\partial_y \mathcal{Q}_{t,h}(x,a) = 0$.
Much of the existing work on one-dimensional convolutions relies heavily on the study of the properties of this hyperbolic Cauchy problem, namely existence, uniqueness and positivity of solution (see \cite{berezansky1998,bloomheyer1994,sousaetal2019b,sousaetal2020} and references therein).

Assume now that $E$ is a compact metric space and the Feller semigroup $\{T_t\}_{t \geq 0}$ has a dense orthogonal set of eigenfunctions $\{\varphi_j\}_{j \in \mathbb{N}}$ such that $\varphi_1 = \mathds{1}$ and $\varphi_j(a) = \|\varphi_j\|_\infty = 1$ for all $j \in \mathbb{N}$. We saw in Proposition \ref{prop:fltc_prodform_kerncharact} that the positivity property
\[
\mathcal{Q}_{t,h}(x,y) := \sum_{j=1}^\infty {1 \over \|\varphi_j\|_2^2} e^{-t \lambda_j} \varphi_j(x) \, \varphi_j(y) \, \langle h, \varphi_j \rangle \geq 0 \qquad \text{for all }\, h \in \mathrm{C}(E) \,\text{ and }\, t > 0.
\]
is crucial in proving the existence of an FLTC for $\{T_t\}$. The function $\mathcal{Q}_{t,h}(x,y)$ is now a solution of $\mathcal{G}_x u = \mathcal{G}_y u$ (where $\mathcal{G}_x$ is the Feller generator $\mathcal{G}$ acting on the variable $x$) such that $\mathcal{Q}_{t,h}(x,a) = (T_t h)(x)$. Moreover, since the point $a$ is a maximizer of all the eigenfunctions $\varphi_j$, the function $\mathcal{Q}_{t,h}(x,y)$ also satisfies (at least formally) the boundary condition $(\nabla_y \mathcal{Q}_{t,h})(x,a) = 0$. (This could be justified e.g.\ by proving that the series can be differentiated term by term. In the next section we will see that this argument can be applied to the Neumann Laplacian on suitable bounded domains of $\mathbb{R}^d$.) This indicates that, as for the one-dimensional problem, the (positivity) properties of the boundary value problem
\begin{equation} \label{eq:fltc_ultrahyp_bvp}
\mathcal{G}_x u = \mathcal{G}_y u, \qquad u(x,a) = u_0(x), \qquad (\nabla_y u)(x,a) = 0
\end{equation}
are related with the problem of constructing a convolution associated with the given strong Feller semigroup.

Consider Examples \ref{exam:fltc_feller_euclidean}--\ref{exam:fltc_feller_boundeddom} or, more generally, any example of a strong Feller semigroup generated by a uniformly elliptic differential operator on $E \subset \mathbb{R}^d$\, ($d>2$). In this context, the principal part of the differential operator $\mathcal{G}_x - \mathcal{G}_y$ has $d$ terms ${\partial^2 \over \partial x_j^2}$ with positive coefficient and $d$ terms ${\partial^2 \over \partial y_j^2}$ with negative coefficient. Such partial differential operators are often said to be of \emph{ultrahyperbolic type} (cf.\ e.g.\ \cite[\S I.5]{petrovsky1954} and \cite[Definition 2.6]{renardyrogers2004}). According to the results of \cite{craigweinstein2009} and \cite[\S VI.17]{courant1962}, the solution for the boundary value problem \eqref{eq:fltc_ultrahyp_bvp} is, in general, not unique. The existing theory on well-posedness of ultrahyperbolic boundary value problems is, in many other respects, rather incomplete; in particular, as far as we know, no maximum principles have been determined for such problems. Adapting the partial differential equation techniques commonly used in the study of one-dimensional convolutions to problems defined on multidimensional spaces is, therefore, a highly nontrivial problem.

\section{Nonexistence results for convolutions on multidimensional domains} \label{sec:nonexist_multidim}

The results of Section \ref{sec:characterizations} show that the existence of an FLTC for a given multidimensional diffusion process depends on two conditions -- the common maximizer property and the positivity of an ultrahyperbolic boundary value problem -- for which there are no reasons to hope that they can be established other than in special cases.
In particular, the common maximizer property is rather counter-intuitive in the multidimensional setting: while in the case of one-dimensional diffusions it is natural that the properties of a Sturm-Liouville operator enforce one of the endpoints of the interval to be a common maximizer, this is no longer natural on bounded domains of $\mathbb{R}^d$ with differentiable boundary or on smooth manifold because one no longer expects that one of the points of the boundary will play a special role.

In fact, as we will see in Subsection \ref{sub:eigenexp_critical_nonex}, under certain conditions one can prove that (reflected) Brownian motions on bounded domains of $\mathbb{R}^d$ or on compact Riemannian manifolds (cf.\ Examples \ref{exam:fltc_feller_boundeddom} and \ref{exam:fltc_feller_riemannian} respectively) do not satisfy the common maximizer property and, therefore, it is not possible to construct an associated FLTC. The preceding Subsection \ref{sub:nonexist_examples} discusses some examples which are useful in better understanding the geometrical properties of the eigenfunctions that are usually encountered in the multidimensional case.

\subsection{Special cases and numerical examples} \label{sub:nonexist_examples}

The first example illustrates the fact that the construction of the convolution becomes trivial if the generator of the multidimensional diffusion decomposes trivially (via separation of variables) into one-dimensional Sturm-Liouville operators for which an associated FLTC exists.

\begin{example}[Neumann eigenfunctions of a $d$-dimensional rectangle] \label{exam:eigenf_rectangle}
Consider the $d$-dimensional rectangle $E = [0,\beta_1] \times \ldots \times [0,\beta_d] \subset \mathbb{R}^d$. The (nonnormalized) eigenfunctions of the Neumann Laplacian on $E$ and the associated eigenvalues are given by
\[
\varphi_{j_1,\ldots,j_d}(x_1,\ldots,x_d) = \prod_{k=1}^d \cos\Bigl(\pi j_k {x_k \over \beta_k}\Bigr), \qquad \lambda_{j_1,\ldots,j_d} = \pi^2 \sum_{k=1}^d {j_k^2 \over \beta_k^2} \qquad (j_1, \ldots, j_d \in \mathbb{N}_0).
\]
These eigenfunctions constitute an orthogonal basis of $L^2(E,dx)$. The point $(0,\ldots,0)$ is, obviously, a maximizer of all the functions $\varphi_{j_1,\ldots,j_d}$, thus the common maximizer property holds. 

The eigenfunctions $\varphi_{j_1,\ldots,j_d}(\bm{\cdot})$ are the product of the eigenfunctions of the Sturm-Liouville operators $-{d^2 \over dx_k^2}$ on the interval $[0,\beta_k]$ (with Neumann condition).
It is known \cite[Example 3.4.6]{bloomheyer1994} that the two-point support convolution $\smash{\delta_x \convccirc{\beta_k} \delta_y}\rule[-0.68\baselineskip]{0pt}{0pt} = {1 \over 2}(\delta_{|x-y|} + \delta_{\beta_k-|\beta_k-x-y|})$ is the convolution of a hypergroup of compact type.
Therefore, it is an FLTC for the semigroup generated by this Sturm-Liouville operator.
Accordingly, one can easily check that the product of the convolutions $([0,\beta_1], \smash{\convccirc{\beta_1}}), \ldots, ([0,\beta_d], \smash{\convccirc{\beta_d}})$, defined as 
\[
(\mu \diamond \nu)(\bm{\cdot}) = \int_E \int_E \bigl((\delta_{x_1} \convccirc{\beta_1} \delta_{y_1}) \otimes \ldots \otimes (\delta_{x_d} \convccirc{\beta_d} \delta_{y_d})\bigr)(\bm{\cdot})\, \mu(dx) \mskip0.5\thinmuskip \nu(dy),
\]
satisfies all the requirements of Definition \ref{def:fltc_fltcdef}, i.e.\ it is an FLTC for the reflected Brownian motion on $E$ (cf.\ Example \ref{exam:fltc_feller_boundeddom}).
\end{example}

Next we present three special cases of circular regions where one can assess whether the common maximizer property holds by analysing the known closed-form expressions for the eigenfunctions.

\begin{example}[Neumann eigenfunctions of disks and balls]
Let $E \subset \mathbb{R}^2$ be the closed disk of radius $R$. The eigenfunctions of the Neumann Laplacian on $E$ are given, in polar coordinates, by
\begin{align*}
\varphi_{0,k}(r,\theta) & = J_0(j_{0,k}' \mskip0.5\thinmuskip \tfrac{r}{R}) \\
\varphi_{m,k,1}(r,\theta) & = J_m(j_{m,k}' \mskip0.5\thinmuskip \tfrac{r}{R}) \cos(m\theta) \\
\varphi_{m,k,2}(r,\theta) & = J_m(j_{m,k}' \mskip0.5\thinmuskip \tfrac{r}{R}) \sin(m\theta)
\end{align*}
where $m,k \in \mathbb{N}$,\, $J_m(\bm{\cdot})$ is the Bessel function of the first kind, and $j_{m,k}'$ stands for the $k$-th (simple) zero of the derivative $J_m'(\bm{\cdot})$ (see \cite[Section 7.2]{komornikloreti2005} and \cite[Proposition 2.3]{helffersundqvist2016}). The corresponding eigenvalues are $\lambda_{0,k} = (j_{0,k}'/R)^2$ (with multiplicity $1$) and $\lambda_{m,k} = (j_{m,k}'/R)^2$ (with multiplicity $2$). It is known from \cite[pp.\ 485, 488]{watson1944}) that for $m \geq 1$ we have $|J_m(j'_{m,k})| > |J_m(x)|$ for all $x > j'_{m,k}$, hence the eigenfunctions $\varphi_{m,k,1}$ and $\varphi_{m,k,2}$ ($m \geq 1$) attain their global maximum on the circle $\bigl\{ r={j'_{m,1} \over j'_{m,k}}R \bigr\}$. This shows, in particular, that no orthogonal basis of $L^2(E,dx)$ composed of Neumann eigenfunctions can satisfy the common maximizer property.

More generally, if $E \subset \mathbb{R}^d$ is a closed $d$-ball with radius $R$, then the eigenfunctions of the Neumann Laplacian on $E$ are
\[
\varphi_{m,k}(r,\theta) = r^{1-{d \over 2}} J_{m-1+{d \over 2}}(c_{m,k} \tfrac{r}{R}) \, H_m(\theta)
\]
where $(r,\theta)$ are hyperspherical coordinates, $m \in \mathbb{N}_0$, $k \in \mathbb{N}$, $H_m$ is a spherical harmonic of order $m$ (see \cite{komornikloreti2005}) and $c_{m,k}$ is the $k$-th zero of the function $\xi \mapsto (1-{d \over 2}) J_{m-1+{d \over 2}}(\xi) + \xi J_{m-1+{d \over 2}}'(\xi)$. The corresponding eigenvalues are $\lambda_{m,k} = c_{m,k}^2$, whose multiplicity is equal to the dimension of the space of spherical harmonics of order $m$. By similar arguments we conclude that the common maximizer property does not hold.
\end{example}

\begin{example}[Neumann eigenfunctions of a circular sector]
Let $E \subset \mathbb{R}^2$ be the sector of angle $\pi \over q$, $E = \{(r\cos\theta, r\sin\theta) \mid 0 \leq r \leq 1, \, 0 \leq \theta \leq {\pi \over q}\}$, where $q \in \mathbb{N}$. The eigenfunctions of the Neumann Laplacian on $E$ and the associated eigenvalues (which have multiplicity $1$, cf.\ \cite{bobkov2018}) are given by \vspace{-2pt}
\[
\varphi_{m,k}(r,\theta) = \cos(qm\theta) J_{qm}(j'_{qm,k}\tfrac{r}{R}), \qquad \lambda_{m,k} = (j'_{qm,k}/R)^2. \vspace{-2pt}
\]
As in the previous example it follows that the global maximizer of $\varphi_{m,k}$ lies in the arc $\bigl\{ r={j'_{qm,1} \over j'_{qm,k}}R \bigr\}$, so that the common maximizer property does not hold.
\end{example}

\begin{example}[Neumann eigenfunctions of a circular annulus]
If $E \subset \mathbb{R}^2$ is the annulus $\{(r,\theta) \mid r_0 \leq r \leq R, \, 0 \leq \theta < 2\pi\}$, where $0 < r_0 < R < \infty$, then the Neumann eigenfunctions on $E$ are
\begin{equation} \label{eq:neumeig_explicit_annulus}
\begin{aligned}
\varphi_{0,k}(r,\theta) & = J_0(c_{0,k} \tfrac{r}{R}) Y_0'(c_{0,k}) - J_0'(c_{0,k}) Y_0(c_{0,k} \tfrac{r}{R}) \\
\varphi_{m,k,1}(r,\theta) & = \bigl(J_m(c_{m,k} \tfrac{r}{R}) Y_m'(c_{m,k}) - J_m'(c_{m,k}) Y_m(c_{m,k} \tfrac{r}{R}) \bigr) \cos(m\theta) \\
\varphi_{m,k,2}(r,\theta) & = \bigl(J_m(c_{m,k} \tfrac{r}{R}) Y_m'(c_{m,k}) - J_m'(c_{m,k}) Y_m(c_{m,k} \tfrac{r}{R}) \bigr) \sin(m\theta)
\end{aligned}
\end{equation}
where $m,k = 1,2,\ldots$,\, $Y_m(\bm{\cdot})$ is the Bessel function of the second kind \cite[\S10.2]{dlmf} and $c_{m,k}$ is the $k$-th zero of the function $\xi \mapsto J_m'(\tfrac{r_0}{R} \mskip0.7\thinmuskip \xi) Y_m'(\xi) - J_m'(\xi) Y_m'(\tfrac{r_0}{R} \mskip0.7\thinmuskip \xi)$. The associated eigenvalues are $\lambda_{m,k} = (c_{m,k}^2/R^2)$. Figure \ref{fig:contour_annulus} presents the contour plots of some of the Neumann eigenfunctions, obtained in two different ways: in panel (a) using the explicit representations \eqref{eq:neumeig_explicit_annulus}, where the constants $c_{m,k}$ are computed numerically with the help of the \texttt{NSolve} function of \emph{Wolfram Mathematica}; and in panel (b) using a numerical approximation of the eigenvalues and eigenfunctions which was computed via the \texttt{NDEigensystem} routine of \emph{Wolfram Mathematica}. Since the eigenvalues $\lambda_{m,k}$ with $m \geq 1$ have multiplicity $2$, the plots obtained by these two approaches differ by a rotation. The results indicate that some of the eigenfunctions (those associated with the first zero $c_{m,1}$) attain their maximum at the outer circle $\{r=R\}$, while other eigenfunctions (those associated with the higher zeros $c_{m,k}$, $k \geq 2$) attain their maximum either at the inner circle $\{r=r_0\}$ or at the interior of the annulus. It is therefore clear that the Neumann eigenfunctions do not satisfy the common maximizer property.

\begin{figure}[t]
 \centering
 \makebox[8cm][c]{
 \renewcommand{\tabcolsep}{1pt}
    	\begin{tabular}{c}    	
    	\vspace*{-2pt}  \includegraphics[width=\textwidth]{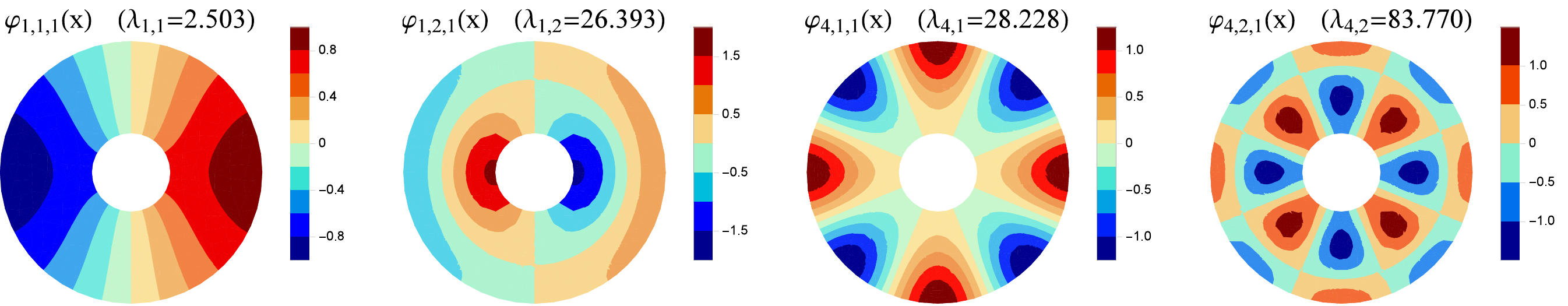}
		\\
		\small \textbf{(a)} Closed form expressions
    	\vspace*{8pt} \\
		\includegraphics[width=\textwidth]{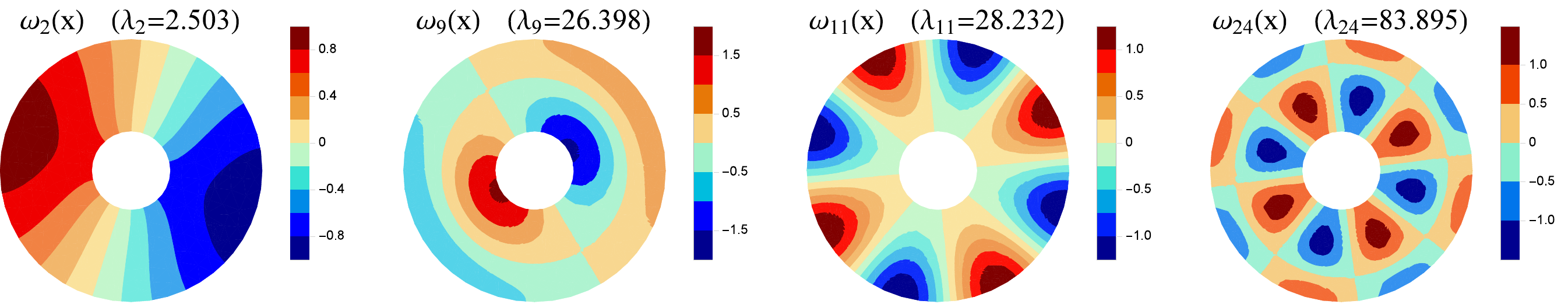}
		\\
		\small \textbf{(b)} Numerical approximation
		\vspace*{-4pt}
    	\end{tabular}
 }
 \caption[Contour plots of the Neumann eigenfunctions of a circular annulus with inner radius $r_0 = 0.3$ and outer radius $R = 1$]{Contour plots of the Neumann eigenfunctions of a circular annulus with inner radius $r_0 = 0.3$ and outer radius $R = 1$. In panel (b), the notation $\omega_k$ refers to the orthogonal eigenfunction associated with the $k$-th largest eigenvalue $\lambda_k$. In both panels the eigenfunctions were normalized so that their $L^2$ norm equals $1$. Similar results were obtained for other values of $r_0 \over R$.}
 \label{fig:contour_annulus}
\end{figure}
\end{example}

There are few other examples of domains of $\mathbb{R}^d$ for which the Neumann eigenfunctions can be computed in closed form. However, in the general case of an arbitrary domain $E \subset \mathbb{R}^2$ it is still possible to find out whether the common maximizer property holds by inspecting the contour plots of the eigenfunctions; these can be computed, for a given bounded domain of $\mathbb{R}^2$, by the same procedure which was used to produce the plots in panel (b) of Figure \ref{fig:contour_annulus}.

This is illustrated in Figures \ref{fig:contour_ellipse_def} and \ref{fig:contour_pentagon_def}, which present the contour plots of the first eigenfunctions of two non-symmetric bounded regions of $\mathbb{R}^2$ with smooth boundary. As we can see, the eigenfunctions attain their maximum values at different points which lie either on the boundary or at the interior of the domain. Note also that the associated eigenvalues are simple, which is unsurprising since the domain has no symmetries (cf.\ comment after Corollary \ref{cor:fltc_prodform_simplspectr_iff}).

\begin{figure}[t]
 \centering
 \makebox[8cm][c]{
 \adjustbox{trim={0\width} {.07\height} {0\width} {0\height},clip}%
  {\includegraphics[width=\textwidth]{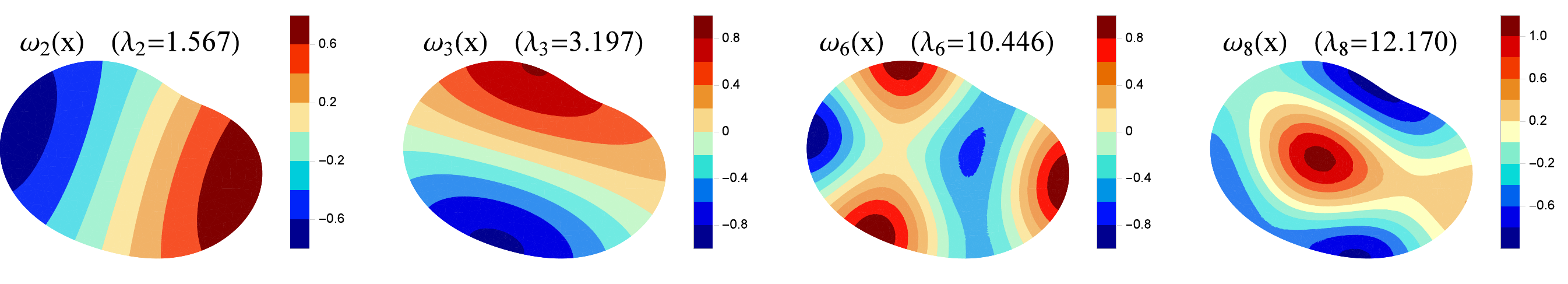}}
 }
 \caption[Contour plots of some eigenfunctions of a region obtained by a non-symmetric deformation of an ellipse]{Contour plots of some eigenfunctions of a region obtained by a non-symmetric deformation of an ellipse. (As above, we denote by $\omega_k$ the Neumann eigenfunction associated with the $k$-th largest eigenvalue $\lambda_k$, and the plots were produced using the \texttt{NDEigensystem} function of \emph{Wolfram Mathematica}.)}
 \label{fig:contour_ellipse_def}
\end{figure}

\begin{figure}[t]
 \centering
 \makebox[8cm][c]{
 \includegraphics[width=\textwidth]{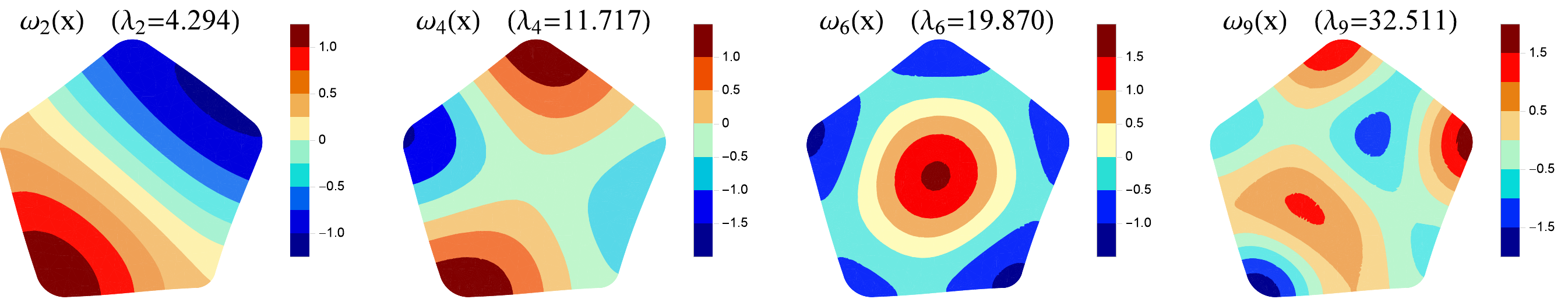}
 }
 \caption{Contour plots of the Neumann eigenfunctions of a region obtained by a non-symmetric deformation of a pentagon with smoothed corners.}
 \label{fig:contour_pentagon_def}
\end{figure}

\begin{remark}[Connection with the hot spots conjecture]
All the examples presented above have the property that \emph{if $\varphi_2$ is an eigenfunction associated with the smallest nonzero Neumann eigenvalue $\lambda_2$, then the maximum and minimum of $\varphi_2$ are attained at the boundary $\partial E$}.
This is the so-called \emph{hot spots conjecture} of J.\ Rauch, which asserts that this property should hold on any bounded domain of $\mathbb{R}^d$.
The physical intuition behind this conjecture is that, for large times, the hottest point on an insulated body with a given initial temperature distribution should converge towards the boundary of the body.

The hot spots conjecture has been extensively studied in the last two decades: it has been shown that the conjecture holds on convex planar domains with a line of symmetry \cite{banuelosburdzy1999,pascu2002}, on convex domains $E \subset \mathbb{R}^2$ with ${\mathrm{diam}(E)^2 \over \mathrm{Area}(E)} < 1.378$ \cite{miyamoto2009} and on any Euclidean triangle \cite{judgemondal2020} (for further positive results see \cite{judgemondal2020} and references therein). On the other hand, some counterexamples have also been found, namely certain domains with holes \cite{burdzy2005}.

The common maximizer property can be interpreted as an extended hot spots conjecture: instead of requiring that the maximum of (the absolute value of) the second Neumann eigenfunction is attained at the boundary, one requires that the maximum of all the eigenfunctions is attained at a common point of the boundary. The negative result of Corollary \ref{cor:fltc_nocommoncrit_smooth} below shows that the location of the hottest point in the limiting distribution (as time goes to infinity) of the temperature of an insulated body depends on the initial temperature distribution.

The common maximizer property and the hot spots conjecture are subtopics of the more general problem of understanding the topological and geometrical structure of Laplacian eigenfunctions, which is the subject of a huge amount of literature. We refer to \cite{grebenkovnguyen2013,jainsamajdar2017} for a survey of known facts, applications and related references.
\end{remark}

\subsection{Eigenfunction expansions, critical points and nonexistence theorems} \label{sub:eigenexp_critical_nonex}

We now proceed to discuss the (failure of the) common maximizer property for reflected Brownian motions on general bounded domains of $\mathbb{R}^d$\, ($d \geq 2$).
Our strategy for disproving the existence of common maximizers is based on two observations.
The first is quite obvious: if $a$ is a common maximizer for the eigenfunctions $\{\varphi_j\}_{j \in \mathbb{N}}$, then it is a common critical point, i.e.\ we have $(\nabla \varphi_j)(a) = 0$ for all $j$.
The second observation is that the usual eigenfunction expansion
\begin{equation} \label{eq:neumeig_usualexpansion}
f = \sum_{j=1}^\infty {1 \over \|\varphi_j\|_2^2} \langle f, \varphi_j\rangle \, \varphi_j
\end{equation}
suggests that the point $a$ will also be a critical point of any function $f$ which is sufficiently regular so that the expansion \eqref{eq:neumeig_usualexpansion} is convergent in the pointwise sense and can be differentiated term by term. Thus if we prove that such pointwise convergence and differentiation is admissible for a class of functions whose derivatives are not restricted to vanish at any given point, then the common maximizer property cannot hold. The next proposition and corollary make this rigorous.

\begin{proposition} \label{prop:fltc_unifconvggrad}
Let $E \subset \mathbb{R}^d$ be the closure of a bounded convex domain. Let $\{X_t\}$ be the reflected Brownian motion on $E$ (Example \ref{exam:fltc_feller_boundeddom}), let $\{\omega_j\}_{j \in \mathbb{N}}$ be an orthonormal basis of $L^2(E) \equiv L^2(E,dx)$ consisting of eigenfunctions of the Neumann Laplacian $-\mathcal{G}^{(2)} \equiv -\Delta_N: \mathcal{D}(\Delta_N) \longrightarrow L^2(E)$ and let $0 \leq \lambda_1 \leq \lambda_2 \leq \lambda_3 \leq \ldots$ be the associated eigenvalues. Let $m \in \mathbb{N}$, $m > {d \over 2} + 1$ and let $h \in H^m(E)$ be a function such that $\Delta^k h \in \mathcal{D}(\Delta_N)$ for $k = 0, 1, \ldots, m-1$. Then 
\begin{equation} \label{eq:fltc_unifconvggrad}
h(x) = \sum_{j=0}^\infty \langle h,\omega_j \rangle \, \omega_j(x) \quad \text{ and } \quad (\nabla h)(x) = \sum_{j=0}^\infty \langle h,\omega_j \rangle (\nabla \omega_j)(x) \quad \text{ for all } x \in E,
\end{equation}
where both series converge absolutely and uniformly on $E$. 
\end{proposition}

\begin{proof}
Let $\{T_t^{(2)}\}_{t \geq 0}$ and $\{\mathcal{R}_\eta^{(2)}\}_{\eta > 0}$ be, respectively, the strongly continuous semigroup and resolvent on $L^2(E)$ generated by the Neumann Laplacian and let $p_t(x,y)$ be the Neumann heat kernel, i.e.\ the transition density of the semigroup $\{T_t^{(2)}\}$. Using the Sobolev embedding theorem \cite[Corollary 6.1]{wloka1987}, one can prove (cf.\ \cite[proof of Theorem 5.2.1]{davies1989}) that the heat kernel is $\mathrm{C}^\infty$ jointly in the variables $(t,x,y) \in (0,\infty) \times E \times E$. Denote by $\partial_{\vec{\bm{v}},x}$ the directional derivative with respect to the variable $x \in \mathbb{R}^d$ in a given direction $\vec{\bm{v}} \in \mathbb{R}^d \setminus \{0\}$. Then there are constants $c_j$ such that the following estimates hold:
\begin{align}
\label{eq:fltc_unifconvggrad_pf1} p_t(x,y) & \leq c_1 t^{-d/2} \exp\biggl(-{|x-y|^2 \over c_2 t}\biggr) \\[4pt]
\label{eq:fltc_unifconvggrad_pf2} |\partial_{\vec{\bm{v}},y} p_t(x,y)| & \leq c_3 t^{-(d+1)/2} \exp\biggl(-{|x-y|^2 \over c_4 t}\biggr).
\end{align}
(The first of these estimates is a basic property of the Neumann heat kernel, see \cite[Theorem 3.1]{basshsu1991}. The second estimate was established in \cite[Lemma 3.1]{wangyan2013}.) Using the basic semigroup identity for the heat kernel, we obtain
\begin{equation} \label{eq:fltc_unifconvggrad_pf3}
|\partial_{\vec{\bm{v}},x} \partial_{\vec{\bm{v}},y} p_t(x,y)| \leq \int_E |\partial_{\vec{\bm{v}},x} p_{t/2}(x,\xi)\, \partial_{\vec{\bm{v}},y} p_{t/2}(\xi,y)| d\xi \leq c_5 t^{-(d+1)} \exp\biggl(-{|x-y|^2 \over 2c_4 t}\biggr).
\end{equation}

Next we recall that \cite[Problem 2.9]{davies1980} \vspace{3pt}
\[
(\mathcal{R}_\alpha^{(2)\!})^k h \equiv (\alpha - \Delta_N)^{-k} h = {1 \over (k-1)!} \int_0^\infty e^{-\alpha t} t^{k-1} \, T_t^{(2)} h\, dt \qquad (h \in L^2(E),\; \alpha > 0,\; k = 1,2,\ldots)
\]
and therefore the $k$-th power $(\mathcal{R}_\alpha^{(2)})^k$ of the resolvent is an integral operator with kernel
\begin{equation} \label{eq:fltc_unifconvggrad_pf4}
G_{\alpha,k}(x,y) = {1 \over (k-1)!} \int_0^\infty e^{-\alpha t} t^{k-1} p_t(x,y) \, dt.
\end{equation}
If $k=2m > d+1$, then using the estimate \eqref{eq:fltc_unifconvggrad_pf1} we see that $G_{\alpha,2m}(x,x) < \infty$ and, furthermore, $G_{\alpha,2m}$ is a continuous function of $(x,y) \in E \times E$. Since $E$ is compact and $(\mathcal{R}_\alpha^{(2)\!})^{2m}: L^2(E) \longrightarrow L^2(E)$ is nonnegative and has a continuous kernel, an application of Mercer's theorem  \cite[Theorem 3.11.9]{simon2015} yields that the kernel $G_{\alpha,2m}$ can be represented by the spectral expansion
\begin{equation} \label{eq:fltc_unifconvggrad_pf5}
G_{\alpha,2m}(x,y) = \sum_{j=1}^\infty {\omega_j(x) \omega_j(y) \over (\alpha + \lambda_j)^{2m}}
\end{equation}
where the series converges absolutely and uniformly in $(x,y) \in E \times E$. (Note that $\bigl((\alpha + \lambda_j)^{-2m}, \omega_j\bigr)$ are the eigenvalue-eigenfunction pairs for $\mathcal{R}_\alpha^{(2)\!}$.) In addition, it follows from \eqref{eq:fltc_unifconvggrad_pf4} and the estimates \eqref{eq:fltc_unifconvggrad_pf2}--\eqref{eq:fltc_unifconvggrad_pf3} that
\[
\partial_{\vec{\bm{v}},x} \partial_{\vec{\bm{v}},y} G_{\alpha,2m}(x,y) = {1 \over (2m-1)!}\int_0^\infty e^{-\alpha t} t^{2m-1} \partial_{\vec{\bm{v}},x} \partial_{\vec{\bm{v}},y} p_t(x,y) \, dt
\]
where the integral converges absolutely and uniformly and defines a continuous function of $(x,y) \in E \times E$. (The function $\partial_{\vec{\bm{v}},y} G_{\alpha,2m}$ is also continuous on $E \times E$.) Using standard arguments (cf.\ \cite[\S21.2, proof of Corollary 3]{naimark1968}), one can then deduce from \eqref{eq:fltc_unifconvggrad_pf5} that 
\begin{equation} \label{eq:fltc_unifconvggrad_pf6}
\partial_{\vec{\bm{v}},x} \partial_{\vec{\bm{v}},y} G_\alpha^{(2m)\mskip-0.8\thinmuskip}(x,y) = \sum_{j=1}^\infty {(\partial_{\vec{\bm{v}}}\omega_j)(x) (\partial_{\vec{\bm{v}}}\omega_j)(y) \over (\alpha + \lambda_j)^{2m}}
\end{equation}
again with absolute and uniform convergence in $(x,y) \in E \times E$.

Let $h \in H^m(E)$ be such that $\Delta^k h \in \mathcal{D}(\Delta_N)$ for $k = 0, 1, \ldots, m-1$, and write $h = \mathcal{R}_\alpha^m g$ where $g := (\alpha - \Delta_N)^m h \in L^2(E)$. Since $m > {d \over 2} + 1$, we have $h \in \mathrm{C}^1(E)$ by the Sobolev embedding theorem. We thus have
\begin{align*}
\sum_{j=0}^\infty \bigl| \langle h,\omega_j \rangle \, \omega_j(x) \bigr|
& = \sum_{j=0}^\infty {|\langle g,\omega_j \rangle| \over (\alpha + \lambda_j)^{m}} |\omega_j(x)| \\
& \leq \biggl(\sum_{j=0}^\infty |\langle g,\omega_j \rangle|^2 \biggr)^{\!{1 \over 2}} \ccdot \biggl(\sum_{j=0}^\infty {|\omega_j(x)|^2 \over (\alpha + \lambda_j)^{2m}} \biggr)^{\!{1 \over 2}} \\
& = \|g\| \ccdot \bigl(G_\alpha^{(2m)\mskip-0.8\thinmuskip}(x,x)\bigr)^{1/2} < \infty
\end{align*}
and similarly
\[
\sum_{j=0}^\infty \bigl| \langle h,\omega_j \rangle (\partial_{\vec{\bm{v}}} \omega_j)(x) \bigr| \leq  \|g\| \ccdot \bigl|\partial_{\vec{\bm{v}},x} \partial_{\vec{\bm{v}},y} G_\alpha^{(2m)\mskip-0.8\thinmuskip}(x,x) \bigr|^{1/2} < \infty.
\]
This shows that the series in the right-hand sides of \eqref{eq:fltc_unifconvggrad} converge absolutely and uniformly in $x \in E$, and the result immediately follows.
\end{proof}

\begin{corollary}[Nonexistence of common critical points] \label{cor:fltc_nocommoncrit_smooth}
Let $m \in \mathbb{N}$, $m > {d \over 2} + 1$ and let $E \subset \mathbb{R}^d$ ($d \geq 2$) be the closure of a bounded convex domain with $\mathrm{C}^{2m+2}$ boundary $\partial E$. Let $\{\omega_j\}_{j \in \mathbb{N}}$ be an orthonormal basis of $L^2(E)$ consisting of eigenfunctions of $-\Delta_N$. Then for each $x_0 \in E$ there exists $j \in \mathbb{N}$ such that $(\nabla \omega_j)(x_0) \neq 0$.
\end{corollary}

\begin{proof}
If $x_0 \in \mathring{E}$, it is clearly possible to choose $h \in \mathrm{C}_\mathrm{c}^\infty(E) \subset \{ u \in H^m(E) \mid u, \Delta u, \ldots, \Delta^{m-1}u \in \mathcal{D}(\Delta_N)\}$ such that $(\nabla h)(x_0) \neq 0$. If $x_0$ belongs to $\partial E$, let $\vec{\bm{v}} \in T_{\partial E} (x_0) \setminus \{0\}$ and choose $\varphi \in \mathrm{C}^\infty(\partial E)$ such that $d\varphi_{x_0}(\vec{\bm{v}}) \neq 0$. Combining the inverse trace theorem for Sobolev spaces \cite[Theorem 8.8]{wloka1987} with the Sobolev embedding theorem, we find that that there exists $h \in H^{2m}(E) \subset \mathrm{C}^1(E)$ such that
\[
h\restrict{\partial E} = \varphi \qquad \text{ and } \qquad \mathrm{Tr}_{\raisebox{-1pt}{\footnotesize$\partial E$}} \Bigl({\partial^j h \over \partial \bm{n}^j}\Bigr) = 0, \quad j=1,2,\ldots,2m-1
\]
where $\bm{n}$ denotes the unit outer normal vector orthogonal to $\partial E$. Consequently, $h$ is such that $(\nabla h)(x_0) \neq 0$ and $h, \Delta h, \ldots, \Delta^{m-1}h \in \mathcal{D}(\Delta_N) = \bigl\{u \in H^2(E) \bigm| \mathrm{Tr}_{\raisebox{-1pt}{\footnotesize$\partial E$}} \bigl({\partial h \over \partial \bm{n}}\bigr) = 0 \bigr\}$. (This characterization of $\mathcal{D}(\Delta_N)$ is well-known, see \cite[Section 10.6.2]{schmudgen2012}.) Therefore, given any $x_0 \in E$ we can apply Proposition \ref{prop:fltc_unifconvggrad} to the function $h$ defined above to conclude that 
\[
\sum_{j=0}^\infty \langle h,\omega_j \rangle (\nabla \omega_j)(x_0) = (\nabla h)(x_0) \neq 0,
\]
which implies that $(\nabla \omega_j)(x_0) \neq 0$ for at least one $j$.
\end{proof}

The conclusions of Proposition \ref{prop:fltc_unifconvggrad} and Corollary \ref{cor:fltc_nocommoncrit_smooth} are also valid for the eigenfunctions of the Laplace-Beltrami operator on a compact Riemannian manifold (Example \ref{exam:fltc_feller_riemannian}):

\begin{proposition}[Nonexistence of common critical points on compact Riemannian manifolds] \label{prop:fltc_nocommoncrit_riemannian}
Let $(E,g)$ be a compact Riemannian manifold (without boundary) of dimension $d$ and $\{\omega_j\}_{j \in \mathbb{N}}$ an orthonormal basis of $L^2(E,\mb{m})$ consisting of eigenfunctions of the Laplace-Beltrami operator on $(E,g)$. Then \eqref{eq:fltc_unifconvggrad} holds for all functions $h \in H^{2m}(E) := \{u \mid u, \Delta u, \ldots, \Delta^m u \in L^2(E,\mb{m})\}$\, ($m \in \mathbb{N},\, m > {d \over 4} + {1 \over 2}$), with the series converging absolutely and uniformly. Furthermore, for each $x_0 \in E$ there exists $j \in \mathbb{N}$ such that $(\nabla \omega_j)(x_0) \neq 0$.
\end{proposition}

\begin{proof}
We know from \cite[Theorem 5.2.1]{davies1989} that the heat kernel $p_t(x,y)$ for the Laplace-Beltrami operator is $\mathrm{C}^\infty$ jointly in the variables $(t,x,y) \in(0,\infty) \times E \times E$. In addition, the heat kernel $p_t(x,y)$ and its gradient satisfy, for $0 < t \leq 1$ and $x, y \in E$, the upper bounds
\[
p_t(x,y) \leq c_1 t^{-d/2} \exp\biggl( -{\mb{d}(x,y)^2 \over c_2 t} \biggr), \qquad |\nabla_y p_t(x,y)| \leq c_3 t^{-(d+1)/2} \exp\biggl( -{\mb{d}(x,y)^2 \over c_4 t} \biggr)
\]
where $\mb{d}$ is the Riemannian distance function. (For the proof see \cite{hsu1999} and \cite[Corollary 15.17]{grigoryan2009}.) Let $U \subset E$ be a coordinate neighbourhood. Arguing as in the proof of Proposition \ref{prop:fltc_unifconvggrad}, we find that for $x,y \in U$ the kernel of the $2m$-th power of the resolvent admits the spectral representation \eqref{eq:fltc_unifconvggrad_pf5} and can be differentiated term by term as in \eqref{eq:fltc_unifconvggrad_pf6}. (The directional derivatives are defined in local coordinates.) We know that for $m > {d \over 4} + {1 \over 2}$ the Sobolev embedding $H^{2m}(E) \subset \mathrm{C}^1(E)$ holds on the Riemannian manifold $E$ \cite[Theorem 7.1]{grigoryan2009}; therefore, the estimation carried out above yields that the expansions \eqref{eq:fltc_unifconvggrad} hold. We have $\partial E = \emptyset$, thus for each $x_0 \in E$ we can choose $h \in \mathrm{C}^\infty(E)$ such that $(\nabla h)(x_0) \neq 0$. As in the proof of Corollary \ref{cor:fltc_nocommoncrit_smooth} it follows that $(\nabla \omega_j)(x_0) \neq 0$ for at least one $j$.
\end{proof}

As noted above, the existence of a common critical point is a necessary condition for the common maximizer property to hold; in turn, this is (under the assumption that the spectrum is simple, cf.\ Corollary \ref{cor:fltc_prodform_simplspectr_iff}) a necessary condition for the existence of an FLTC. Therefore, the following nonexistence theorem is a direct consequence of the preceding results.

\begin{theorem} \label{thm:fltc_nonexistence}
Let $\{T_t\}_{t \geq 0}$ be either the Feller semigroup on a bounded domain $E \subset \mathbb{R}^d$ with $\mathrm{C}^{2m+2}$ boundary ($m > {d \over 2} + 1$) associated with the reflected Brownian motion on $E$ or the Feller semigroup associated with the Brownian motion on a compact Riemannian manifold. Assume that the operator $T_1^{(2)}$ has simple spectrum. Then there exists no FLTC for the semigroup $\{T_t\}$.
\end{theorem}

This theorem is not applicable to regular polygons and other domains which are invariant under reflection or rotation (i.e.\ under the natural action of a dihedral group), as this invariance enforces the presence of eigenvalues with multiplicity greater than $1$. On the other hand, we know that the eigenspaces on such symmetric domains can be associated to the different symmetry subspaces of the irreducible representations of the dihedral group \cite{helfferetal2002}. In most cases, the multiplicity of all the eigenspaces corresponding to the one-dimensional irreducible representations is equal to $1$ \cite{marrocospereira2015}; therefore, an adaptation of the proofs presented above should allow us to establish the nonexistence of common critical points among the eigenfunctions associated to the one-dimensional eigenspaces.

The nonexistence theorem established above strongly depends on the discreteness of the spectrum of the generator of the Feller process. Extending Theorem \ref{thm:fltc_nonexistence} to Brownian motions on unbounded domains on $\mathbb{R}^d$ or on noncompact Riemannian manifolds is a challenging problem, as these diffusions generally have a nonempty continuous spectrum. We leave this topic for future research.

\section*{Acknowledgements}

The first and third authors were partially supported by CMUP, which is financed by national funds through FCT – Fundação para a Ciência e a Tecnologia, I.P., under the project with reference\linebreak UIDB/00144/2020. The first author was also supported by the grant PD/BD/135281/2017, under the FCT PhD Programme UC|UP MATH PhD Program. The second author was partially supported by the project CEMAPRE/REM – UIDB/05069/2020 – financed by FCT through national funds.

\linespread{1.125}
\renewcommand{\bibname}{References} 
\begin{small}

\end{small}

\end{document}